\documentclass[a4paper]{amsart}
\usepackage{amssymb, enumitem}
\usepackage[all]{xy}
\usepackage{hyperref, aliascnt, graphicx}

\newtheorem{lma}{Lemma}[section]

\newaliascnt{thmCt}{lma}
\newtheorem{thm}[thmCt]{Theorem}
\aliascntresetthe{thmCt}

\newaliascnt{corCt}{lma}
\newtheorem{cor}[corCt]{Corollary}
\aliascntresetthe{corCt}

\newaliascnt{prpCt}{lma}
\newtheorem{prp}[prpCt]{Proposition}
\aliascntresetthe{prpCt}

\newtheorem*{thm*}{Theorem}
\newtheorem*{cor*}{Corollary}
\newtheorem*{prop*}{Proposition}

\newtheorem{thmx}{Theorem}

\newtheorem{corx}[thmx]{Corollary}

\theoremstyle{definition}

\newaliascnt{pgrCt}{lma}
\newtheorem{pgr}[pgrCt]{}
\aliascntresetthe{pgrCt}

\newaliascnt{dfnCt}{lma}
\newtheorem{dfn}[dfnCt]{Definition}
\aliascntresetthe{dfnCt}

\newaliascnt{rmkCt}{lma}
\newtheorem{rmk}[rmkCt]{Remark}
\aliascntresetthe{rmkCt}

\newaliascnt{exaCt}{lma}
\newtheorem{exa}[exaCt]{Example}
\aliascntresetthe{exaCt}

\def\today{\number\day~\ifcase\month\or   January\or February\or
   March\or April\or May\or June\or   July\or August\or September\or
   October\or November\or December\fi\   \number\year}

\newcommand{\NN}{{\mathbb{N}}}
\newcommand{\RR}{{\mathbb{R}}}

\newcommand{\id}{{\mathrm{id}}}
\newcommand{\ev}{{\mathrm{ev}}}
\newcommand{\tr}{\mathrm{t}}
\newcommand{\tensProj}{\widehat{\otimes}}
\newcommand{\ca}{$C^*$-algebra}
\newcommand{\wkStar}{\operatorname{\mathrm{w}^*\!-\!}}
\newcommand{\weakStar}{weak${}^*$}
\newcommand{\freeVar}{\_\,}
\newcommand{\andSep}{\,\,\,\text{ and }\,\,\,}
\newcommand{\ltArensProd}{\square}
\newcommand{\rtArensProd}{\lozenge}
\newcommand{\ltArensI}[1]{#1^{\backprime}}
\newcommand{\ltArensII}[1]{#1^{\backprime\backprime}}
\newcommand{\ltArensIII}[1]{#1^{\backprime\backprime\backprime}}
\newcommand{\rtArensI}[1]{#1^{\prime}}
\newcommand{\rtArensII}[1]{#1^{\prime\prime}}
\newcommand{\rtArensIII}[1]{#1^{\prime\prime\prime}}
\DeclareMathOperator{\kernel}{ker}
\newcommand{\vcong}{\rotatebox[origin=c]{90}{$\cong$}}

\newcommand{\tensPredual}[2]{\overline{\otimes}_{#2}}
\newcommand{\Bdd}{{\mathcal{L}}}
\newcommand{\Approx}{{\mathcal{A}}}
\newcommand{\Cpct}{{\mathcal{K}}}

\title[Preduals and complementation of $\Bdd(X,Y)$]{Preduals and complementation of spaces of bounded linear operators}
\date{\today}

\author[Eusebio Gardella]{Eusebio Gardella}
\address{Eusebio Gardella. Mathematisches Institut, Universit\"at M\"unster, Einsteinstr.~62, 48149 M\"unster, Germany.}
\email{gardella@uni-muenster.de}
\urladdr{www.math.uni-muenster.de/u/gardella/}
\author{Hannes Thiel}
\address{Hannes Thiel. Mathematisches Institut, Universit\"at M\"unster, Einsteinstr.~62, 48149 M\"unster, Germany.}
\email{hannes.thiel@uni-muenster.de}
\urladdr{www.math.uni-muenster.de/u/hannes.thiel/}\thanks{The first named author was partially supported by a Postdoctoral Research Fellowship from the Humboldt Foundation.
Both authors were partially supported by the Deutsche Forschungsgemeinschaft (SFB 878 \emph{Groups, Geometry \& Actions}).
Part of the research was conducted at the Mittag-Leffler institute during the 2016 program on Classification of Operator Algebras: Complexity, Rigidity, and Dynamics.}

\begin{document}

\begin{abstract}
For Banach spaces $X$ and $Y$, we establish a natural bijection between preduals of $Y$ and preduals of $\Bdd(X,Y)$ that respect the right $\Bdd(X)$-module structure.
If $X$ is reflexive, it follows that there is a unique predual making $\Bdd(X)$ into a dual Banach algebra.
This removes the condition that $X$ have the approximation property in a result of Daws.

We further establish a natural bijection between projections that complement $Y$ in its bidual and $\Bdd(X)$-linear projections that complement $\Bdd(X,Y)$ in its bidual.
It follows that $Y$ is complemented in its bidual if and only if $\Bdd(X,Y)$ is (either as a module or as a Banach space).

Our results are new even in the well-studied case of isometric preduals.
\end{abstract}

\maketitle

\section{Introduction}

Given a Banach space $X$, we study preduals of $\Bdd(X)$ in relation to preduals of $X$.
The most satisfactory results in this direction are obtained when $\Bdd(X)$ is regarded not only as a Banach space, but as a Banach algebra (more precisely, as a right $\Bdd(X)$-module).
Thus, we will focus on preduals of $\Bdd(X)$ for which right multiplication is \weakStar{} continuous.
One of our results implies that such preduals are in one-to-one correspondence with preduals of $X$;
see \autoref{thm:A} below.

In fact, the methods here developed apply equally well to analyzing preduals of $\Bdd(X,Y)$, considered with its natural right $\Bdd(X)$-module structure.
Moreover, shifting the attention from $\Bdd(X)$ to $\Bdd(X,Y)$ clarifies the different roles that the domain space $X$ and target space $Y$ play for properties of $\Bdd(X,Y)$.
For example, and somewhat interestingly, it turns out that preduals of $\Bdd(X,Y)$ are induced by preduals of $Y$, while preduals of $X$ do not play any role.
It is thus natural to address the more general problem of studying preduals of $\Bdd(X,Y)$ compatible with the right $\Bdd(X)$-action, in relation to preduals of $Y$.

A (concrete) predual of a Banach space $Y$ is a closed subspace $F\subseteq Y^*$ inducing a canonical isomorphism between $F^*$ and $Y$;
see \autoref{dfn:predualBSp:concretePredualBSp}.
If this isomorphism is isometric, we call $F$ an \emph{isometric} predual.
Existence and uniqueness of preduals has been extensively studied in various settings;
we refer to the survey article \cite{God89IsometricPredualsSurvey} by Godefroy, and the references therein.
For instance, by Sakai's theorem, a \ca{} has an isometric predual if and only if it is a von Neumann algebra, and the isometric predual of a von Neumann algebra is unique.
Similarly, by \cite[Proposition~5.10]{GodSap88DualitySpOpsSmoothNorms}, if $X$ is reflexive, then $\Bdd(X)$ has a unique isometric predual.

Our focus is on preduals that are not necessarily isometric.
Such preduals are not as well-studied: results do not simply carry over from isometric to general preduals, and new phenomena appear in the general setting.
For example, Sakai's theorem is no longer true for not-necessarily isometric preduals, and neither is the above mentioned result of Godefroy and Saphar;
see \autoref{exa:refl:notStrUnique}.

Every predual $F\subseteq Y^*$ induces a natural predual $X\tensPredual{X}{Y}F$ of $\Bdd(X,Y)$, which makes the right action of $\Bdd(X)$ on $\Bdd(X,Y)$ \weakStar{} continuous;
see \autoref{dfn:predualBXY:predualBXY} and \autoref{prp:predualBXY:predualBXY}.
We say that such a predual makes $\Bdd(X,Y)$ a right dual $\Bdd(X)$-module.
The converse is the main result of this paper:

\begin{thmx}[{See \autoref{prp:predualBXY:correspondencePreduals}}]
\label{thm:A}
Let $X$ and $Y$ be Banach spaces with $X\neq\{0\}$.
Then assigning to a predual $F\subseteq Y^*$ the predual $X\tensPredual{X}{Y}F\subseteq\Bdd(X,Y)^*$ induces a natural one-to-one correspondence between:
\begin{enumerate}
\item[(a)]
Concrete preduals of $Y$.
\item[(b)]
Concrete preduals of $\Bdd(X,Y)$ making it a right dual $\Bdd(X)$-module.
\end{enumerate}
\end{thmx}

We obtain this result by studying projections $\Bdd(X,Y)^{**}\to\Bdd(X,Y)$.
In general, considering $Y$ as a subspace of its bidual $Y^{**}$, it is folklore that there is a natural one-to-one correspondence between preduals of $Y$ and projections $Y^{**}\to Y$ with \weakStar{} closed kernel;
see \autoref{prp:predualBSP:summary}.
Every projection $\pi\colon Y^{**}\to Y$ induces a natural projection $r_\pi\colon\Bdd(X,Y)^{**}\to\Bdd(X,Y)$ that is a right $\Bdd(X)$-module map;
see \autoref{prp:compl:q_r_from_pi}.
The converse is the foundation for most new results of this paper:

\begin{thmx}[{See \autoref{prp:compl:complBidual}}]
\label{thm:B}
Let $X$ and $Y$ be Banach spaces with $X\neq\{0\}$.
Assigning to a projection $\pi\colon Y^{**}\to Y$ the projection $r_\pi\colon\Bdd(X,Y)^{**}\to\Bdd(X,Y)$ induces a natural one-to-one correspondence between:
\begin{enumerate}
\item[(a)]
Projections $Y^{**}\to Y$.
\item[(b)]
Projections $\Bdd(X,Y)^{**}\to\Bdd(X,Y)$ that are right $\Bdd(X)$-module maps.
\end{enumerate}
Moreover, the kernel of $\pi$ is \weakStar{} closed if and only if so is the kernel of $r_\pi$.
\end{thmx}

\autoref{thm:A} can be deduced from \autoref{thm:B} by noticing that a predual of $\Bdd(X,Y)^*$ makes it a right dual $\Bdd(X)$-module if and only if 
the associated projection $\Bdd(X,Y)^{**}\to\Bdd(X,Y)$ is a right $\Bdd(X)$-module map.
Another interesting consequence of Theorem~B is:

\begin{corx}[{See \autoref{prp:compl:charCompl}}]
\label{cor:C}
Let $X$ and $Y$ be Banach spaces with $X\neq\{0\}$.
Then the following are equivalent:
\begin{enumerate}
\item
$Y$ is complemented in its bidual.
\item
$\Bdd(X,Y)$ is complemented in its bidual as a Banach space.
\item
$\Bdd(X,Y)$ is complemented in its bidual as a right $\Bdd(X)$-module.
\end{enumerate}
\end{corx}

As a further application, in \autoref{prp:refl:charYRefl_BXY} we characterize reflexivity of $Y$ in terms of $\Bdd(X,Y)$.
A Banach algebra $A$ with a predual is said to be a left (right) dual Banach algebra if the left (right) multiplication by a fixed element in $A$ is \weakStar{} continuous.
If $A$ is both a left and a right dual Banach algebra, then it is called a dual Banach algebra; see \cite{Run02BookAmen,Daw04PhD,Daw07DualBAlgReprInj}.
In the case $X=Y$, we obtain:

\begin{corx}
\label{cor:D}
For a Banach space $X$, the following are equivalent:
\begin{enumerate}
\item
$X$ is reflexive.
\item
$\Bdd(X)$ has a predual making it a (left) dual Banach algebra.
\item
$\Bdd(X)$ is complemented in its bidual as a left $\Bdd(X)$-module.
\item
There exists a projection $r\colon\Bdd(X)^{**}\to\Bdd(X)$ that is multiplicative for the left (equivalentely, for the right, or for both) Arens product on $\Bdd(X)^{**}$.
\end{enumerate}
\end{corx}

The equivalence between (1) and (2) above was previously obtained by Daws in~\cite{Daw04PhD}, but the algebraic characterization of reflexivity of $X$ in terms of the existence of a multiplicative projection $\Bdd(X)^{**}\to\Bdd(X)$ is new.

Next, we obtain uniqueness results for preduals of $\Bdd(X,Y)$ assuming that $Y$ is reflexive;
see \autoref{prp:refl:unique_BXY}.
Specializing to the case $X=Y$, we obtain:

\begin{corx}[{See \autoref{prp:refl:unique_BX}}]
\label{cor:E}
If $X$ is reflexive, then $X\tensProj X^*$ is the unique predual of $\Bdd(X)$ making it a right dual Banach algebra.
\end{corx}

In particular, if $X$ is reflexive, and if $A$ is any right dual Banach algebra with a (not necessarily isometric) Banach algebra isomorphism $\varphi\colon\Bdd(X)\to A$, then $\varphi$ and $\varphi^{-1}$ are automatically \weakStar{} continuous.
This improves a result of Daws, \cite[Theorem~4.4]{Daw07DualBAlgReprInj}, by removing an assumption while also strengthening the conclusion;
see the comments after \autoref{prp:refl:unique_BX}.

The results in this paper are the foundations for two other works of the authors; see \cite{GarThi17arX:ExtendingRepr} and \cite{GarThi16arX:ReprConvLq}.


\subsection*{Acknowledgements}

We thank Sven Raum for helpful discussions on \autoref{sec:alpha}.
We thank Matthew Daws, Gilles Godefroy, Bill Johnson, Philip Spain and Stuart White for valuable electronic correspondence.

\subsection*{Notation}

Fix Banach spaces $X$ and $Y$.
We let $\Bdd(X,Y)$ denote the Banach space of bounded, linear maps $X\to Y$.
We write $\Cpct(X,Y)$ and $\Approx(X,Y)$ for the compact and approximable operators (the norm-closure of the finite-rank operators $X\to Y$), respectively.
We write $\Bdd(X)$ for $\Bdd(X,X)$, and similarly with $\Cpct(X)$ and $\Approx(X)$.
For $x\in X$, we write $\ev_x\colon \Bdd(X,Y)\to Y$ for the evaluation map at $x$.
We let $X^{**}$ denote the bidual of $X$, and we let $\kappa_X\colon X\to X^{**}$ denote the canonical, isometric embedding.
We decorate a variable with one or two stars to indicate that it belongs to the dual or bidual of some space, for example $x^*\in X^*$ or $y^{**}\in Y^{**}$.
We denote the adjoint of a bounded, linear map $a\colon X\to Y$ by $a^\tr\colon Y^*\to X^*$.

\section{Preliminaries}

Let $X, Y$ and $Z$ denote Banach spaces; 
and let $A$ and $B$ denote Banach algebras.


\subsection{Preduals of Banach spaces, algebras and modules}
\label{sec:predualBSp}


\begin{dfn}
\label{dfn:predualBSp:predualBSp}
A \emph{predual} of $X$ is pair $(Y,\delta)$ consisting of a Banach space $Y$ and an isomorphism $\delta\colon X\to Y^*$.
The predual is called \emph{isometric} if $\delta$ is an isometric isomorphism.
Two preduals $(Y_1,\delta_1)$ and $(Y_2,\delta_2)$ are \emph{(isometrically) equivalent} if there exists an (isometric) isomorphism $\varphi\colon Y_2\to Y_1$ such that $\delta_1=\varphi^*\circ\delta_2$.
\end{dfn}

The following definition is inspired by \cite{DawHaySchWhi12ShiftInvPreduals}.

\begin{dfn}
\label{dfn:predualBSp:concretePredualBSp}
A \emph{concrete predual} of $X$ is a closed subspace $F\subseteq X^*$ such that $\iota_F^\tr\circ\kappa_X$ is an isomorphism, where $\iota_F\colon F\to X^*$ denotes the inclusion map.
Given a concrete predual $F\subseteq X^*$, we define maps
\[
\delta_F := \iota_F^\tr\circ\kappa_X \colon X\to F^*, \andSep
\pi_F := \delta_F^{-1}\circ\kappa_F^\tr\circ\delta_F^{\tr\tr}\colon X^{**}\to X.
\]
If $\delta_F$ is isometric, we call $F$ a \emph{concrete isometric predual} of $X$.
If no confusion may arise, we will call a concrete predual simply a predual.
\end{dfn}

\begin{rmk}
Observe that a predual $F$ is isometric if and only if $\|\delta_F^{-1}\|\leq 1$.
\end{rmk}

\begin{pgr}
\label{pgr:predualBsp:predualsConcrete}
Let $F\subseteq X^*$ be a predual.
The pair $(F,\delta_F)$ is a predual of $X$ in the sense of \autoref{dfn:predualBSp:predualBSp}.
Conversely, let $(Y,\delta)$ be a predual of $X$.
Let $F_\delta$ denote the image of the map $\delta^\tr\circ\kappa_Y\colon Y\to X^*$.
Then $F_\delta$ is a concrete predual of $X$, and
$(Y,\delta)$ is equivalent to $(F_\delta,\delta_{F_\delta})$.
It follows that two preduals $(Y_1,\delta_1)$ and $(Y_2,\delta_2)$ are equivalent if and only if $F_{\delta_1}=F_{\delta_2}$.
Hence, concrete preduals of $X$ naturally correspond to equivalence classes of preduals as in \autoref{dfn:predualBSp:predualBSp}.

If $(Y_1,\delta_1)$ and $(Y_2,\delta_2)$ are isometric preduals, then they are equivalent if and only if they are isometrically equivalent.
If $(Y,\delta)$ is an isometric predual, then $F_\delta$ is a concrete isometric predual, whence $(F_\delta,\delta_{F_\delta})$ is isometrically equivalent to $(Y,\delta)$.
\end{pgr}

\begin{pgr}
\label{pgr:predualBSp:predualProj}
Let $F\subseteq X^*$ be a predual with inclusion map $\iota_F\colon F\to X^*$.
Then $\kappa_F^\tr\circ\delta_F^{\tr\tr}=\iota_F^\tr$, and therefore $\pi_F =\delta_F^{-1}\circ\iota_F^\tr$.
It follows that $\pi_F\circ\kappa_X=\id_X$ and $\|\pi_F\|=\|\delta_F^{-1}\|$.
Further, the kernel of $\pi_F$ is \weakStar{} closed in $X^{**}$.

Conversely, let $\pi\colon X^{**}\to X$ be a projection with \weakStar{} closed kernel.
Set
\[
F_\pi := \big\{ x^*\in X^* : \langle x^*,x\rangle=0 \text{ for all } x\in\kernel(\pi) \big\},
\]
which is called the \emph{pre-annihilator} of $\kernel(\pi)$.
Then $F_\pi$ is a predual of $X$ and $\pi_{F_\pi}=\pi$.
Conversely, the projection $\pi_F$ associated to $F\subseteq X^*$ satisfies $F_{\pi_F}=F$.
\end{pgr}

We summarize the mentioned correspondences.

\begin{prp}
\label{prp:predualBSP:summary}
Let $X$ be a Banach space.
Then there are natural one-to-one correspondences between the following classes:
\begin{enumerate}
\item
Concrete preduals of $X$.
\item
Equivalence classes of preduals as in \autoref{dfn:predualBSp:predualBSp}.
\item
Projections $X^{**}\to X$ with \weakStar{} closed kernel.
\end{enumerate}
Moreover, given a predual $(Y,\delta)$ of $X$ we have
\[
\|\pi_{F_\delta}\|
= \min\left\{ \|\gamma\| \|\gamma^{-1}\| : (Z,\gamma) \text{ is predual of $X$ equivalent to } (Y,\delta) \right\}.
\]
\end{prp}

\begin{dfn}
\label{pgr:predualBSp:unique}
Assume that $X$ has a (isometric) predual.
One says that $X$ has a \emph{unique (isometric) predual} if for any (isometric) preduals $(Y_1,\delta_1)$ and $(Y_2,\delta_2)$ the spaces $Y_1$ and $Y_2$ are (isometrically) isomorphic.
Equivalently, any two (isometric) preduals $F_1,F_2\subseteq X^*$ are (isometrically) isomorphic as Banach spaces.

Further, $X$ has a \emph{strongly unique predual} if any two preduals of $X$ are equivalent in the sense of \autoref{dfn:predualBSp:predualBSp}.
Analogously, $X$ has a \emph{strongly unique isometric predual} if $X$ has an isometric predual and if any two isometric preduals of $X$ are equivalent (or equivalently, isometrically equivalent).

It is not known if there exists a Banach space with unique but not strongly unique predual;
see \cite[Problem~2]{God89IsometricPredualsSurvey}.
\end{dfn}


In this work, we will mostly be interested in preduals of Banach algebras, and more generally, Banach modules.
The following is inspired in \cite[Definition~4.1]{Spa15ReprCaInDualBAlg}; note that we do not require 
the predual of a dual Banach algebra to be isometric.

\begin{dfn}
\label{dfn:predualMod:dualMod}
Let $A$ be a Banach algebra.
A \emph{left dual $A$-module} is a left $A$-module $\mathcal{E}$ together with a (concrete) predual $F\subseteq \mathcal{E}^*$ such that for each $a\in A$, its (left) action $L_a$ on $\mathcal{E}$ is \weakStar{} continuous.
Right dual $B$-modules are defined analogously.
A \emph{dual $A$-$B$-bimodule} is an $A$-$B$-bimodule with a predual giving it the structure of both a left dual $A$-module and a right dual $B$-module.

Given a predual for $A$, one says that $A$ is a \emph{(left, right) dual Banach algebra} if it is a (left, right) dual module over itself.
One says that $A$ is a \emph{dual Banach algebra} if it is a dual $A$-$A$-bimodule.
\end{dfn}

\begin{pgr}
\label{pgr:dualMod:unextDualMod}
Let $\mathcal{E}$ be an $A$-$B$-bimodule.
Then $\mathcal{E}^*$ has a natural $B$-$A$-bimodule structure, with left action of an element $b\in B$ on $\mathcal{E}^*$ given by the transpose of the right action of $b$ on $\mathcal{E}$, and analogously for the right action of $A$ on $\mathcal{E}^*$.
Similarly, $\mathcal{E}^{**}$ has an $A$-$B$-bimodule structure, and the map $\kappa_\mathcal{E}\colon \mathcal{E}\to \mathcal{E}^{**}$ is an $A$-$B$-bimodule map.
\end{pgr}

The map $\pi_F$ from \autoref{dfn:predualBSp:concretePredualBSp} can be used to characterize whether
an $A$-module is a dual $A$-module, as we show next.
Since the proof is elementary, we include it here for the convenience of the reader.

\begin{prp}
\label{prp:predualMod:charDualMod}
Let $\mathcal{E}$ be a left $A$-module, and let $F\subseteq \mathcal{E}^*$ be a predual.
Then the following are equivalent:
\begin{enumerate}
\item
$(\mathcal{E},F)$ is a left dual $A$-module.
\item
$F$ is a right sub-$A$-module of $\mathcal{E}^*$, that is, $FA\subseteq F$.
\item
$F$ has a right $A$-module structure such that $\delta_F$ is a left $A$-module map.
\item
$\pi_F$ is a left $A$-module map.
\end{enumerate}
Analogous results hold for right modules and bimodules.
In particular, if $\mathcal{E}$ is an $A$-$B$-bimodule, then $(\mathcal{E},F)$ is a dual $A$-$B$-bimodule if and only if $F$ is a sub-$B$-$A$-bimodule 
of $\mathcal{E}^*$, which is equivalent to $\pi_F$ being an $A$-$B$-bimodule map.
\end{prp}
\begin{proof}
We show the equivalence for a left $A$-module $\mathcal{E}$.
Let $m\colon A\times \mathcal{E}\to \mathcal{E}$ be the bilinear map defining the left $A$-module structure on $\mathcal{E}$.
Given $a\in A$, the map $L_a=m(a,\freeVar)$ is \weakStar{} continuous for the \weakStar{} topology induced by $F$ if and only if its transpose $L_a^\tr\colon \mathcal{E}^*\to \mathcal{E}^*$ leaves $F$ invariant.
The right action of $a$ on $\mathcal{E}^*$ is given by $L_a^\tr$.
Therefore, $L_a$ is \weakStar{} continuous if and only if $Fa\subseteq F$.
This implies the equivalence between~(1) and~(2).

Let us show that~(2) implies~(3).
We give $F$ the right $A$-module structure it inherits as a right sub-$A$-module of $\mathcal{E}^*$.
Then the inclusion map $\iota\colon F\to \mathcal{E}^*$ is a right $A$-module map.
Hence, $\iota^\tr\colon \mathcal{E}^{**}\to F^*$ is a left $A$-module map.
Since $\kappa_\mathcal{E}$ is also a left $A$-module map, we deduce that $\delta_F$ ($=\iota^\tr\circ\kappa_\mathcal{E}$) is a left $A$-module map.

Let us show that~(3) implies~(4).
Since $\kappa_F\colon F\to F^{**}$ is a right $A$-module map, the transpose $\kappa_F^\tr\colon F^{***}\to F^*$ is a left $A$-module map.
Since $\delta_F$ is a left $A$-module map, so are $\delta_F^{\tr\tr}$ and $\delta_F^{-1}$.
Hence, $\pi_F$ ($=\delta_F^{-1}\circ\kappa_F^\tr\circ\delta_F^{\tr\tr}$) is a left $A$-module map.

To show that~(4) implies~(2), let $\xi^*\in F\subseteq\mathcal{E}^*$ and $a\in A$.
To verify $L_a^\tr(\xi^*)\in F$, we recall from \autoref{pgr:predualBSp:predualProj} that $F$ is the preannihilator of $\kernel(\pi_F)$.
Given $\xi^{**}\in\kernel(\pi_F)\subseteq\mathcal{E}^{**}$, we have $L_a^{\tr\tr}(\xi^{**})\in\kernel(\pi_F)$ by assumption, and thus
\[
\langle L_a^\tr(\xi^*),\xi^{**}\rangle 
= \langle \xi^*,L_a^{\tr\tr}(\xi^{**})\rangle
= 0,
\]
and hence $L_a^\tr(\xi^*)\in F$, as desired.
\end{proof}

\subsection{Biduals of Banach algebras and modules}

In this subsection, we review some classical notions that originated in the groundbreaking works 
of Arens \cite{Are51AdjBilinOp,Are51OpsInducedFctClasses}; see also~\cite[Section~1.7]{Daw04PhD}.
\autoref{prp:dualMod:extDualMod} seems not to have appeared in the literature before.

\begin{pgr}
Let $m\colon X\times Y\to Z$ be a bilinear map.
Arens introduced two procedures to extend $m$ to a bilinear map $X^{**}\times Y^{**}\to Z^{**}$;
see \cite{Are51AdjBilinOp}.
One first defines bilinear maps $\ltArensI{m}\colon Z^*\times X\to Y^*$ and $\rtArensI{m}\colon Y\times Z^*\to X^*$ by
\[
\left\langle \ltArensI{m}(z^*,x),y \right\rangle
= \left\langle z^*, m(x,y) \right\rangle\quad\text{ and }\quad
\left\langle \rtArensI{m}(y,z^*),x \right\rangle
= \left\langle z^*, m(x,y) \right\rangle,
\]
for all $x\in X$, $y\in Y$, and $z^*\in Z^*$.

Applying the same construction again to $\ltArensI{m}$ and $\rtArensI{m}$, one obtains bilinear maps $\ltArensII{m}\colon Y^{**}\times Z^*\to X^*$ and $\rtArensII{m}\colon Z^*\times X^{**}\to Y^*$ given by
\[
\left\langle \ltArensII{m}(y^{**},z^*),x \right\rangle
= \left\langle y^{**}, \ltArensI{m}(z^*,x) \right\rangle\quad\text{ and }\quad
\left\langle \rtArensII{m}(z^*,x^{**}),y \right\rangle
= \left\langle x^{**}, \rtArensI{m}(y,z^*) \right\rangle,
\]
for all $x\in X$, $x^{**}\in X^{**}$, $y\in Y$, $y^{**}\in Y^{**}$ and $z^*\in Z^*$.

Applying the same procedure once again to $\ltArensII{m}$ and $\rtArensII{m}$, one obtains bilinear maps $\ltArensIII{m},\rtArensIII{m}\colon X^{**}\times Y^{**}\to Z^{**}$ given by
\begin{align*}
\left\langle \ltArensIII{m}(x^{**},y^{**}),z^* \right\rangle
= \left\langle x^{**}, \ltArensII{m}(y^{**},z^*) \right\rangle, \
\left\langle \rtArensIII{m}(x^{**},y^{**}),z^* \right\rangle
= \left\langle y^{**}, \rtArensII{m}(z^*,x^{**}) \right\rangle,
\end{align*}
for all $x^{**}\in X^{**}$, $y^{**}\in Y^{**}$ and $z^*\in Z^*$.
\end{pgr}

The following standard fact will be needed in the sequel.

\begin{lma}
\label{prp:dualMod:ArensTranspWkCts}
Let $m\colon X\times Y\to Z$ be a bilinear map.
Then $\ltArensI{m}, \ltArensII{m}$ and $\ltArensIII{m}$ are \weakStar{} continuous in their first variables,
and $\rtArensI{m}, \rtArensII{m}$ and $\rtArensIII{m}$ are \weakStar{} continuous in their second variables.
\end{lma}

\begin{dfn}
\label{dfn:actionsDualMod}
Let $\mathcal{E}$ be a left $A$-module, with module structure given by the bilinear map $m\colon A\times \mathcal{E}\to \mathcal{E}$.
Given $\xi^*\in \mathcal{E}^*$, $\xi^{\ast\ast}\in \mathcal{E}^{**}$ and $a^{**}\in A^{**}$, we set
\[
\xi^* a^{**} = \rtArensII{m}(\xi^*,a^{**}), \quad
a^{**}\ltArensProd \xi^{**} = \ltArensIII{m}(a^{**},\xi^{**}), \andSep
a^{**}\rtArensProd \xi^{**} = \rtArensIII{m}(a^{**},\xi^{**}).
\]

Analogously, if $\mathcal{F}$ is a right $B$-module, with module structure given by the map $n\colon \mathcal{F}\times B\to \mathcal{F}$, and given $\eta^*\in \mathcal{F}^*$, $\eta^{**}\in \mathcal{F}^{**}$ and $b^{**}\in B^{**}$, we set
\[
b^{**}\eta^*  = \ltArensII{n}(b^{**},\eta^*), \quad
\eta^{**}\ltArensProd b^{**} = \ltArensIII{n}(\eta^{**},b^{**}), \andSep
\eta^{**}\rtArensProd b^{**} = \rtArensIII{n}(\eta^{**},b^{**}).
\]

When $\mathcal{E}=A$, we obtain two products $\ltArensProd$ and $\rtArensProd$ on $A^{\ast\ast}$, called the 
\emph{left} and \emph{right Arens products}, respectively. 
\end{dfn}

The following is a straightforward extension of \cite[Theorem~2]{CabGarVil00ExtMultilinearOps}. 

\begin{prp}
\label{prp:dualMod:extDualMod}
Let $\mathcal{E}$ be a left $A$-module.
Then
\[
\xi^* (a^{**}\rtArensProd b^{**}) = (\xi^* a^{**})b^{**}, \ \
(a^{**}\ltArensProd b^{**})\ltArensProd \xi^{**} = a^{**}\ltArensProd (b^{**}\ltArensProd\xi^{**}),
\]
\[
\andSep  (a^{**}\rtArensProd b^{**})\rtArensProd \xi^{**} = a^{**}\rtArensProd (b^{**}\rtArensProd\xi^{**}),
\]
for all $\xi^*\in \mathcal{E}^*$, $\xi^{**}\in \mathcal{E}^{**}$, $a^{**},b^{**}\in A^{**}$.
In particular, $\mathcal{E}^{*}$ is naturally a right $(\mathcal{A}^{\ast\ast},\rtArensProd)$-module;
and $\mathcal{E}^{**}$ has natural structures as a left $(A^{\ast\ast},\ltArensProd)$-module as well as a left $(\mathcal{A}^{\ast\ast},\rtArensProd)$-module.

Analogously, if $\mathcal{F}$ is a right $B$-module, then
\[
(a^{**}\rtArensProd b^{**}) \eta^*  =  a^{**}(b^{**}\eta^*), \ \
\eta^{**}\ltArensProd(a^{**}\ltArensProd b^{**})  =  (\eta^{**}\ltArensProd a^{**}) \ltArensProd b^{**}, 
\]
\[
\andSep \eta^{**}\rtArensProd(a^{**}\rtArensProd b^{**}) = (\eta^{**}\rtArensProd b^{**})\rtArensProd b^{**},
\]
for all $\eta^*\in \mathcal{F}^*$, $\eta^{**}\in \mathcal{F}^{**}$, $a^{**},b^{**}\in A^{**}$.
\end{prp}

\begin{rmk}
\label{rmk:dualMod:extDualMod}
Let $\mathcal{E}$ be a $A$-$B$-bimodule. 
By \autoref{prp:dualMod:extDualMod}, $\mathcal{E}^*$ is a left $(B^{**},\ltArensProd)$-module and a right $(A^{**},\rtArensProd)$-module. 
However, these actions are not necessarily compatible and $\mathcal{E}^*$ need not be a $(B^{**},\ltArensProd)$-$(A^{**},\rtArensProd)$-bimodule.
Similarly, the left $(A^{**},\rtArensProd)$-module structure on $\mathcal{E}^{**}$ need not be compatible with the right $(B^{**},\ltArensProd)$-module structure.
For example, a unital Banach algebra $A$ is Arens regular (that is, $\ltArensProd=\rtArensProd$ on $A^{**}$) if and only if $A^*$ is a $(A^{**},\ltArensProd)$-$(A^{**},\rtArensProd)$-bimodule.
Indeed, for $a^{**},b^{**},c^{**}\in A^{**}$ and $\xi^*\in A^*$, we have
\[
\langle (b^{**}\ltArensProd a^{**})\rtArensProd c^{**},\xi^*\rangle
= \langle b^{**}\ltArensProd a^{**}, c^{**}\xi^*\rangle
= \langle a^{**}, (c^{**}\xi^*)b^{**}\rangle,
\]
and analogously
\[
\langle b^{**}\ltArensProd (a^{**}\rtArensProd c^{**}),\xi^*\rangle
= \langle a^{**}\rtArensProd c^{**}, \xi^*b^{**}\rangle
= \langle a^{**}, c^{**}(\xi^*b^{**})\rangle.
\]
Hence, if $A$ is Arens regular, then $(c^{**}\xi^*)b^{**}=c^{**}(\xi^*b^{**})$.
Conversely, to show that $b^{**}\ltArensProd c^{**}=b^{**}\rtArensProd c^{**}$, consider $a^{**}=1_A$.
\end{rmk}

Given a left $A$-module map $\varphi\colon \mathcal{E}\to \mathcal{F}$, it is shown in 
\cite[Lemma~3]{CabGarVil00ExtMultilinearOps} that $\varphi^{\tr\tr}$ is a left $(A^{**},\ltArensProd)$-module map.
We now generalize this result.

\begin{prp}
\label{prp:dualMod:extDualModFunctorial}
Let $\mathcal{E}$ and $\mathcal{F}$ be left $A$-modules, and
let $\varphi\colon \mathcal{E}\to \mathcal{F}$ be a left $A$-module map.
Then the transpose $\varphi^\tr\colon \mathcal{F}^*\to \mathcal{E}^*$ is a right $(A^{**},\rtArensProd)$-module map.
Further, the bitranspose $\varphi^{\tr\tr}\colon \mathcal{E}^{**}\to \mathcal{E}^{**}$ is both a left $(A^{**},\ltArensProd)$-module map and a left $(A^{**},\rtArensProd)$-module map.
Analogous results hold for right modules and bimodules.
\end{prp}
\begin{proof}
Let the left $A$-module structures on $\mathcal{E}$ and $\mathcal{F}$ be given by bilinear maps $m_\mathcal{E}\colon A\times \mathcal{E}\to \mathcal{E}$ and $m_\mathcal{F}\colon A\times \mathcal{F}\to \mathcal{F}$, respectively.
To show that $\varphi^\tr$ is a right $(A^{**},\rtArensProd)$-module map, let $\eta^{\ast}\in \mathcal{F}^*$ and $a^{**}\in A^{**}$.
Choose a net $(a_i)_{i\in I}$ in $A$ such that $a^{**}=\wkStar\lim_{i\in I}\kappa_A(a_i)$.
Using \autoref{prp:dualMod:ArensTranspWkCts} and that $\varphi^\tr$ is \weakStar{} continuous at the second step, and using that $\varphi^\tr$ is a right $A$-module map at the fourth step, we obtain
\begin{align*}
\varphi^\tr(\eta^{\ast} a^{**})
&= \varphi^\tr\big( \rtArensII{m_\mathcal{F}}(\eta^{\ast},\wkStar\lim_{i\in I}\kappa_A(a_i)) \big)
= \wkStar\lim_{i\in I} \varphi^\tr\big( \rtArensII{m_\mathcal{F}}(\eta^{\ast},\kappa_A(a_i)) \big) \\
&= \wkStar\lim_{i\in I} \varphi^\tr(\eta^{\ast} a_i)
= \wkStar\lim_{i\in I} \varphi^\tr(\eta^{\ast}) a_i
= \wkStar\lim_{i\in I} \rtArensII{m_\mathcal{E}}(\varphi^\tr(\eta^{\ast}),\kappa_A(a_i)) \\
&= \rtArensII{m_\mathcal{E}}(\varphi^\tr(\eta^{\ast}),\wkStar\lim_{i\in I}\kappa_A(a_i))
= \varphi^\tr(\eta^{\ast})a^{**}.
\end{align*}

To show that $\varphi^{\tr\tr}$ is a left $(A^{**},\rtArensProd)$-module map, let $a^{**}\in A^{**}$, let $\xi^{**}\in \mathcal{E}^{**}$, and let $\eta^{\ast}\in \mathcal{F}^*$.
Using that $\varphi^\tr$ is a right $(A^{**},\rtArensProd)$-module map at the third step, we get
\begin{align*}
\langle \varphi^{\tr\tr}(a^{**}\rtArensProd\xi^{**}), \eta^{\ast} \rangle
&= \langle a^{**}\rtArensProd\xi^{**}, \varphi^{\tr}(\eta^{\ast}) \rangle
= \langle \xi^{**}, (\varphi^{\tr}(\eta^*))a^{**} \rangle
= \langle \xi^{**}, \varphi^{\tr}(\eta^* a^{**}) \rangle \\
&= \langle \varphi^{\tr\tr}(\xi^{**}), \eta^{\ast} a^{**} \rangle
= \langle a^{**}\rtArensProd\varphi^{\tr\tr}(\xi^{**}), \eta^* \rangle.
\end{align*}

Finally, to show that $\varphi^{\tr\tr}$ is a left $(A^{**},\ltArensProd)$-module map, let $a^{**}\in A^{**}$ and $\xi^{**}\in \mathcal{E}^{**}$.
Choosing a net $(a_i)_{i\in I}$ in $A$ with $a^{**}=\wkStar\lim_{i\in I}\kappa_A(a_i)$, we argue as above and obtain
\begin{align*}
\varphi^{\tr\tr}(a^{**}\ltArensProd\xi^{**})
&= \varphi^{\tr\tr}\big( \ltArensIII{m_\mathcal{E}}(\wkStar\lim_{i\in I}\kappa_A(a_i),\xi^{**}) \big)
= \wkStar\lim_{i\in I} \varphi^{\tr\tr}(a_i\xi^{**})\\
&= \wkStar\lim_{i\in I} a_i\varphi^{\tr\tr}(\xi^{**}) 
= \wkStar\lim_{i\in I} \ltArensIII{m_\mathcal{F}}(\kappa_A(a_i),\varphi^{\tr\tr}(\xi^{**}))\\
&= \ltArensIII{m_\mathcal{F}}(a^{**},\varphi^{\tr\tr}(\xi^{**}))
= a^{**}\ltArensProd\varphi^{\tr\tr}(\xi^{**}).\qedhere
\end{align*}
\end{proof}

\section{The map \texorpdfstring{$\alpha_{X,Y}\colon\Bdd(X,Y)^{**}\to\Bdd(X,Y^{**})$}{alpha-XY:L(X,Y)**->L(X,Y**)}}
\label{sec:alpha}

Let $X$ and $Y$ be Banach spaces.
The right $\Bdd(X)$-module structure of $\Bdd(X,Y)$ induces a right $(\Bdd(X)^{**},\rtArensProd)$-module structure on $\Bdd(X,Y)^{**}$;
see \autoref{pgr:alpha:bimod-BXY}.
In \autoref{dfn:alpha:alphaXY}, we construct a natural map $\alpha_{X,Y}\colon\Bdd(X,Y)^{**}\to\Bdd(X,Y^{**})$ and show in
\autoref{prp:alpha:alpha-bimod-reduced} that $\alpha_{X,Y}$ respects the module structures, in the sense that
\[
\alpha_{X,Y}(a^{**}\rtArensProd f^{**}) = \alpha_{X,Y}(a^{**})\alpha_X(f^{**})
\]
for all $a^{**}\in\Bdd(X)^{**}$ and $f^{**}\in\Bdd(X,Y)^{**}$.

For $X=Y$, we conclude that the map $\alpha_X\colon\Bdd(X)^{**}\to\Bdd(X,X^{**})$ is multiplicative when $\Bdd(X)^{**}$ is equipped with the \emph{right} Arens product $\rtArensProd$;
see \autoref{prp:alpha:alpha-bimod-reduced}.

\begin{pgr}
\label{pgr:alpha:bimod-BXY}
\emph{$\Bdd(X,Y)$ as $\Bdd(Y)$-$\Bdd(X)$-bimodule.} Given $a\in\Bdd(X)$, we define the right action of $a$ on $\Bdd(X,Y)$ by $fa = f\circ a$ for all $f\in\Bdd(X,Y)$.
Analogously, given $b\in\Bdd(Y)$, we define the left action of $b$ by $bf = b\circ f$ for all $f\in\Bdd(X,Y)$.
It follows from \autoref{prp:dualMod:extDualMod} that $\Bdd(X,Y)^{**}$ has a natural left $(\Bdd(Y)^{**},\ltArensProd)$- and right $(\Bdd(X)^{**},\rtArensProd)$-module structure.
If $\Bdd(Y)$ or $\Bdd(Y)$ is Arens regular (for example, if $X$ or $Y$ is finite dimensional), then these module structures are compatible and $\Bdd(X,Y)^{**}$ is a $(\Bdd(Y)^{**},\ltArensProd)$-$(\Bdd(X)^{**},\rtArensProd)$-bimodule.
In general, however, these left and right module structures are not compatible;
see \autoref{rmk:dualMod:extDualMod}.
\end{pgr}

\begin{pgr}
\label{pgr:prelim:tensProj}
We denote by $X\otimes Y$ the algebraic (that is, uncompleted) tensor product of $X$ and $Y$.
Recall that the \emph{projective cross norm} of $t\in X\otimes Y$ is defined as
\[
\|t\|_{\pi}
= \inf\left\{ \sum_{k=1}^n \|x_k\|\|y_k\|\colon t=\sum_{k=1}^n x_k\otimes y_k \right\}.
\]
We have $\|x\otimes y\|_{\pi}=\|x\|\|y\|$ for every simple tensor $x\otimes y\in X\otimes Y$.
The projective tensor product $X\tensProj Y$ of $X$ and $Y$ is defined as the completion of $X\otimes Y$ in $\|\cdot\|_{\pi}$.
We refer to \cite{Rya02IntroTensProdBSp} for more on the rich theory of tensor products of Banach spaces.

It is well known that there is a natural, isometric isomorphism between $(X\tensProj Y)^*$ and $\Bdd(X,Y^*)$ given as follows:
An operator $a\in\Bdd(X,Y^*)$ defines a functional on $X\otimes Y$, given on simple tensors by $\langle a, x\otimes y \rangle = \langle a(x), y\rangle$.
The resulting functional on $X\tensProj Y$ has norm $\|a\|$.
Thus, for every $t\in X\otimes Y$ we have
\[
\|t\|_{\pi}
= \sup \big\{ \langle a, t\rangle : a\in\Bdd(X,Y^*), \|a\|\leq 1 \big\}.
\]
\end{pgr}

\begin{pgr}
\label{pgr:alpha:sliceMaps}
\emph{Slice maps.}
Given $x\in X$, we consider the map $x\otimes\freeVar \colon Y\to X\tensProj Y$ that sends $y\in Y$ to $x\otimes y$.
Given $x^*\in X^*$, we consider the slice map $R_{x^*}\colon X\tensProj Y\to Y$ which on simple tensors is given by $R_{x^*}(x\otimes y)=\langle x^*,x\rangle y$, for $x\in X$ and $y\in Y$.
Then $R_{x^*}$ is linear and $\|R_{x^*}\|=\|x^*\|$.
\end{pgr}

\begin{pgr}
\label{pgr:alpha:theta}
\emph{Rank-one operators.}
Given $y\in Y$ and $x^*\in X^*$, we let $\theta_{y,x^*}\in \Bdd(X,Y)$ be given by $\theta_{y,x^*}(x)=\langle x^*,x \rangle y$ for all $x\in X$.
For $a\in\Bdd(X)$ and $b\in\Bdd(Y)$, we have
\[
\theta_{y,x^*}a=\theta_{y,a^\tr x^*},\quad
b\theta_{y,x^*}=\theta_{by,x^*}, \quad
\theta_{y,x^*}^\tr=\theta_{x^*,\kappa_Y(y)}, \ \mbox{ and } \ 
\|\theta_{y,x^*}\|=\|y\|\cdot\|x^*\|.
\]

Given $x^*\in X^*$, define $\Theta_{x^*}\colon Y\to\Bdd(X,Y)$ by $\Theta_{x^*}(y)=\theta_{y,x^*}$ for all $y\in Y$.
Then $\Theta_{x^*}$ is linear and $\|\Theta_{x^*}\|=\|x^*\|$.
If $\|x^*\|=1$, then $\Theta_{x^*}$ is in fact isometric.
\end{pgr}

\begin{lma}
\label{lma:evThetawkstercts}
Let $x\in X$ and $x^*\in X^*$.
Then $\ev_x\colon\Bdd(X,Y^*)\to Y^*$ is the transpose of the map $x\otimes\freeVar\colon Y\to X\tensProj Y$, and $\Theta_{x^*}\colon Y^*\to\Bdd(X,Y^*)$ is the transpose of the slice map $R_{x^*}\colon X\tensProj Y\to Y$.
In particular, $\ev_x$ and $\Theta_{x^*}$ are \weakStar{} continuous.
\end{lma}
\begin{proof}
Both are routine computations, and we omit them.
\end{proof}


The following propositon immediately implies \cite[Theorem~4.2]{Spa15ReprCaInDualBAlg} and \cite[Proposition~4.2.1]{Daw04PhD}: 
$X$ is reflexive if and only if $\Bdd(X^*)$ is a left dual Banach algebra.

\begin{prp}
\label{prp:alpha:bimod-BXYd-wkStar}
Let $X$ and $Y$ be Banach spaces.
We identify $\Bdd(X,Y^*)$ with $(X\tensProj Y)^*$, and we consider $\Bdd(X,Y^*)$ with the induced \weakStar{} topology.

Let $a\in\Bdd(X)$ and $b\in\Bdd(Y^*)$.
Then the action of $a$ on $\Bdd(X,Y^*)$ is \weakStar{} continuous.
The action of $b$ on $\Bdd(X,Y^*)$ is \weakStar{} continuous if and only if there exists $c\in\Bdd(Y)$ such that $b=c^\tr$.
\end{prp}
\begin{proof}
There is a left $\Bdd(X)$-module structure on $X\tensProj Y$ given by $\bar{a}(x\otimes y)=\bar{a}(x)\otimes y$ for $\bar{a}\in\Bdd(X)$, and a simple tensor $x\otimes y\in X\tensProj Y$.
We have
\[
\langle fa, x\otimes y \rangle
= \langle f(a(x)), y \rangle
= \langle f, a(x)\otimes y \rangle
\]
for all $f\in \Bdd(X,Y^*)$.
It follows that the action of $a$ on $\Bdd(X,Y^*)$ is given as the transpose of the action of $a$ on $X\tensProj Y$, whence it is \weakStar{} continuous.

Similarly, given $c\in\Bdd(Y)$, there is a bounded, linear map $X\tensProj Y\to X\tensProj Y$ that maps a simple tensor $x\otimes y\in X\tensProj Y$ to $x\otimes c(y)$.
Then
\[
\langle c^\tr f, x\otimes y \rangle
= \langle c^\tr(f(x)), y \rangle
= \langle f(x), c(y) \rangle
= \langle f, x\otimes c(y) \rangle,
\]
which shows that the action of $c^\tr$ is \weakStar{} continuous.

Let $L_b\colon\Bdd(X,Y^*)\to\Bdd(X,Y^*)$ denote the left action of $b$ on $\Bdd(X,Y^*)$.
Choose $x\in X$ and $x^*\in X^*$ with $\langle x^*,x \rangle = 1$.
We claim that $\ev_x\circ L_b\circ\Theta_{x^*}\colon Y^*\to Y^*$ coincides with $b$.
Indeed, for every $y^*\in Y^*$, we have
\[
(\ev_x\circ L_b\circ\Theta_{x^*})(y^*)
= (b\Theta_{x^*}(y^*))(x)
= b(\theta_{y^*,x^*}(x))
= b(\langle x^*,x\rangle y^*)
= b(y^*).
\]
Now assume that $L_b$ is \weakStar{} continuous.
By \autoref{lma:evThetawkstercts}, the maps $\ev_x$ and $\Theta_{x^*}$ are \weakStar{} continuous as well.
It follows that $b$, as a map $Y^*\to Y^*$, is \weakStar{} continuous.
This implies that there exists $c\in\Bdd(Y)$ such that $b=c^\tr$, as desired.
\end{proof}


\begin{dfn}
\label{dfn:alpha:psi}
Let $X$ and $Y$ be Banach spaces.
Let $\mu_{X,Y}\colon Y^*\tensProj X\to X\tensProj Y^*$ be the (unique) isometric isomorphism that satisfies $\mu_{X,Y}(y^*\otimes x)=x\otimes y^*$ for all $x\in X$ and $y^*\in Y^*$.
Identifying $\Bdd(X,Y^{**})$ with $(X\tensProj Y^*)^*$ and $\Bdd(Y^*,X^*)$ with $(Y^*\tensProj X)^*$, we let $\psi_{X,Y}\colon \Bdd(X,Y^{**}) \to \Bdd(Y^*,X^*)$ be the map induced by the transpose of $\mu_{X,Y}$.
We write $\psi_X$ for $\psi_{X,X}$.
\end{dfn}

\begin{prp}
\label{prp:alpha:psi}
Let $X$ and $Y$ be Banach spaces. 
Then $\psi_{X,Y}$ is a \weakStar{} continuous, isometric isomorphism. For $f\in\Bdd(X,Y^{**})$ and $g\in\Bdd(Y^*,X^*)$, we have
\[
\psi_{X,Y}(f)= f^\tr\circ\kappa_{Y^*}\colon Y^*\to X^* \andSep
\psi_{X,Y}^{-1}(g)=g^\tr\circ\kappa_X\colon X\to Y^{**}.
\]
\end{prp}
\begin{proof}
It is clear that $\psi_{X,Y}$ is an isometric isomorphism.
Given $f\in\Bdd(X,Y^{**})$ and a simple tensor $y^*\otimes x\in Y^*\otimes X$, we have
\begin{align*}
\left\langle \psi_{X,Y}(f), y^*\otimes x \right\rangle
&= \left\langle f, x\otimes y^* \right\rangle 
= \left\langle f(x), y^* \right\rangle
= \left\langle x, f^\tr(\kappa_{Y^*}(y^*)) \right\rangle \\
&= \left\langle f^\tr\circ\kappa_{Y^*}, y^*\otimes x \right\rangle.
\end{align*}
It follows that $\psi_{X,Y}(f)= f^\tr\circ\kappa_{Y^*}$, as desired.

Given $g\in\Bdd(Y^*,X^*)$, the following computation shows that $\psi_{X,Y}^{-1}(g)=g^\tr\circ\kappa_X$:
\[
\psi_{X,Y}(g^\tr\circ\kappa_X)
= (g^\tr\circ\kappa_X)^\tr\circ\kappa_{Y^*}
= \kappa_X^\tr\circ g^{\tr\tr}\circ\kappa_{Y^*}
= g.\qedhere
\]
\end{proof}

\begin{pgr}
\label{pgr:alpha:BAlgBXXdd}
Given $a,b\in\Bdd(X,X^{**})$, we define their product as $ab=\kappa_{X^*}^\tr\circ a^{\tr\tr}\circ b$.
It is straightforward to check that this gives $\Bdd(X,X^{**})$ the structure of a unital Banach algebra, the unit being $\kappa_X$.

We define a natural $\Bdd(Y,Y^{**})$-$\Bdd(X,X^{**})$-bimodule structure on $\Bdd(X,Y^{**})$ by
\[
fa=\kappa_{Y^*}^\tr\circ f^{\tr\tr}\circ a, \andSep bf=\kappa_{Y^*}^\tr\circ b^{\tr\tr}\circ f,
\]
for $f\in\Bdd(X,Y^{**})$, $a\in\Bdd(X,X^{**})$, and $b\in\Bdd(Y,Y^{**})$.
\end{pgr}

\begin{prp}
\label{prp:alpha:psiXXdd}
Let $X$ and $Y$ be Banach spaces, let $f\in\Bdd(X,Y^{**})$, let $a\in\Bdd(X,X^{**})$, and let $b\in\Bdd(Y,Y^{**})$.
Then
\[
\psi_{X,Y}(fa)=\psi_X(a)\psi_{X,Y}(f) \andSep
\psi_{X,Y}(bf)=\psi_{X,Y}(f)\psi_Y(b).
\]
In particular, the map $\psi_X\colon \Bdd(X,X^{**}) \to \Bdd(X^*)$ is a \weakStar{} continuous, isometric anti-isomorphism of Banach algebras.
\end{prp}
\begin{proof}
Using \autoref{prp:alpha:psi} at the second step, and using that $f^{\tr\tr\tr} \circ \kappa_{Y^*}^{\tr\tr}=\kappa_{X^*} \circ f^\tr$ at the fourth step, we get
\begin{align*}
\psi_{X,Y}(fa)
&= \psi_{X,Y}(\kappa_{Y^*}^\tr\circ f^{\tr\tr}\circ a)
= (\kappa_{Y^*}^\tr\circ f^{\tr\tr}\circ a)^\tr\circ\kappa_{Y^*}
= a^\tr \circ f^{\tr\tr\tr} \circ \kappa_{Y^*}^{\tr\tr} \circ \kappa_{Y^*} \\
&= a^\tr \circ \kappa_{X^*} \circ f^\tr \circ \kappa_{Y^*}
= \psi_X(a)\psi_{X,Y}(f),
\end{align*}
as desired.
The second identity is shown analogously.
\end{proof}

\begin{pgr}
\label{pgr:alpha:gammaXY}
\emph{The map $\gamma_{X,Y}$.}
Given Banach spaces $X$ and $Y$, we let $\gamma_{X,Y}\colon\Bdd(X,Y)\to\Bdd(X,Y^{**})$ be the isometric map given by $\gamma_{X,Y}(f) =\kappa_Y\circ f$, for all $f\in\Bdd(X,Y)$.
We write $\gamma_X$ for $\gamma_{X,X}$.
We give $\Bdd(X,Y)$ the natural $\Bdd(Y)$-$\Bdd(X)$-bimodule structure from \autoref{pgr:alpha:bimod-BXY}.
Let $f\in\Bdd(X,Y)$.
Given $a\in\Bdd(X)$, we have
\[
\gamma_{X,Y}(fa)
= \kappa_Y\circ f\circ a
= \kappa_{Y^*}^\tr\circ(\kappa_Y\circ f)^{\tr\tr} \circ\kappa_X\circ a
= \gamma_{X,Y}(f)\gamma_X(a).
\]
Similarly, $\gamma_{X,Y}(bf) = \gamma_Y(b)\gamma_{X,Y}(f)$ for all $b\in \Bdd(Y)$.
It follows that $\gamma_X=\gamma_{X,X}\colon\Bdd(X)\to\Bdd(X,X^{**})$ is a homomorphism of Banach algebras.
It is easy to check that the composition $\psi_{X,Y}\circ\gamma_{X,Y}\colon\Bdd(X,Y)\to\Bdd(Y^*,X^*)$ is given by the transpose map $f\mapsto f^\tr$.
\end{pgr}

\begin{rmk}
\label{rmk:alpha:BAlgBXXdd}
Let $a\in\Bdd(X,X^{**})$ and $b\in\Bdd(Y,Y^{**})$.
Note that the left action of $b$ on $\Bdd(X,Y^{**})$ is \weakStar{} continuous by \autoref{prp:alpha:bimod-BXYd-wkStar}.
Moreover, since $c=\psi_X(\gamma_X(c))$ for all $c\in \Bdd(X)$, it follows that  
the right action of $a$ on $\Bdd(X,Y^{**})$ is \weakStar{} continuous if and only if there exists $c\in\Bdd(X)$ such that $a=\gamma_X(c)$.

In particular, and in view of \autoref{prp:alpha:psiXXdd}, while $\Bdd(X^*)$ is \emph{right} dual Banach algebra, $\Bdd(X,X^{**})$ is \emph{left} dual Banach algebra.
\end{rmk}

\begin{pgr}
\label{pgr:alpha:bimod-BXYdd-full}
\emph{$\Bdd(X,Y^{**})$ as a left $(\Bdd(Y)^{**},\ltArensProd)$- and a right $(\Bdd(X)^{**},\rtArensProd)$-module.}
The tensor product $X\tensProj Y^*$ has a natural $\Bdd(X)$-$\Bdd(Y)$-bimodule structure given by
\[
a(x\otimes y^*) = a(x)\otimes y^* \andSep
(x\otimes y^*)b = x\otimes b^\tr(y^*),
\]
for all $a\in\Bdd(X)$, $b\in\Bdd(Y)$, $x\in X$, and $y^*\in Y^*$.
Since $(X\tensProj Y^*)^*\cong\Bdd(X,Y^{**})$ canonically, it follows that $\Bdd(X,Y^{**})$ has both a left $(\Bdd(Y)^{**},\ltArensProd)$- and a right $(\Bdd(X)^{**},\rtArensProd)$-module structure;
see \autoref{prp:dualMod:extDualMod}.
However, these left and right module structures are not necessarily compatible;
see \autoref{rmk:dualMod:extDualMod}.

Let $f\in\Bdd(X,Y^{**})$.
Given $a\in\Bdd(X)$, it is straightforward to check that $f\kappa_{\Bdd(X)}(a) = f\gamma_X(a)$.
It follows that the right action of $\kappa_{\Bdd(Y)}(a)$ on $\Bdd(X,Y^{**})$ is given by precomposing with $a$, which agrees with the right action of $\gamma_X(a)$.

A similar argument shows that $\kappa_{\Bdd(Y)}(b)f = b^{\tr\tr}\circ f = \gamma_Y(b)f$ for all $b\in\Bdd(Y)$.
\end{pgr}

\begin{pgr}
\label{pgr:alpha:omegaXY}
\emph{The map $\omega_{X,Y}$.}
We let
$\omega_{X,Y}\colon X\tensProj Y^* \to\Bdd(X,Y)^*$,
be the natural, contractive, linear map
determined by $\omega_{X,Y}(x\otimes y^*)(f)=\langle f(x), y^*\rangle$, for all $x\in X$, all $y^*\in Y^*$ and all $f\in\Bdd(X,Y)$.
(It is a routine exercise to show that $\omega_{X,Y}$ is well-defined and contractive.)
\end{pgr}

\begin{dfn}
\label{dfn:alpha:alphaXY}
Let $X$ and $Y$ be Banach spaces.
Identifying $\Bdd(X,Y^{**})$ with the dual of $X\tensProj Y^*$, we define the map
\[
\alpha_{X,Y}\colon\Bdd(X,Y)^{**} \to \Bdd(X,Y^{**})
\]
as the transpose of the map $\omega_{X,Y}$.
We write $\alpha_X$ for $\alpha_{X,X}$.
\end{dfn}

\begin{lma}
\label{pgr:alpha:charAlpha}
Let $f^{**}\in\Bdd(X,Y)^{**}$ and let $x\in X$.
Then 
\[
\alpha_{X,Y}(f^{**})(x) = \ev_x^{\tr\tr}(f^{**}), \andSep
\alpha_{X,Y}\circ\kappa_{\Bdd(X,Y)}=\gamma_{X,Y}.
\]
\end{lma}
\begin{proof}
To check the first identity, let $r_x\colon Y^*\to X\tensProj Y^*$ be given by 
$r_x(y^*)=x\otimes y^*$ for all $y^*\in Y^*$.
It is easily verified that $\omega_{X,Y}\circ\ r_x = \ev_x^\tr$.
It follows that $r_x^\tr\circ\alpha_{X,Y} = \ev_x^{\tr\tr}$ and that 
$f(x)=r_x^\tr(f)$ for all $f\in\Bdd(X,Y^{**})$.
Taking $f=\alpha_{X,Y}(f^{**})$, we get
\[
\alpha_{X,Y}(f^{**})(x)
= r_x^\tr(\alpha_{X,Y}(f^{**}))
= \ev_x^{\tr\tr}(f^{**}).
\]

To check the second identity, let $f\in\Bdd(X,Y)$.
Using the first identity at the first step, we conclude that 
\[
(\alpha_{X,Y}\circ\kappa_{\Bdd(X)})(f)(x)
= \ev_x^{\tr\tr}(\kappa_{\Bdd(X)}(f))
= \kappa_Y(f(x))
= \gamma_{X,Y}(f)(x).
\qedhere
\]
\end{proof}

In the following two lemmas, we use the module structures described in
Paragraphs 3.9 (for the right-hand sides) and 3.13 (for the left-hand sides).

\begin{lma}
\label{prp:alpha:bimod-BXYdd-relation}
Let $a^{**}\in\Bdd(X)^{**}$, let $f\in\Bdd(X,Y^{**})$, let $b^{**}\in\Bdd(Y)^{**}$, and let $g\in\Bdd(X,Y)$.
Then
$fa^{**} = f\alpha_X(a^{**})$ and $b^{**}\gamma_{X,Y}(g) = \alpha_Y(b^{**})\gamma_{X,Y}(g)$.
\end{lma}
\begin{proof}
Choose a net $(a_i)_{i\in I}$ in $\Bdd(X)$ such that $a^{**}=\wkStar\lim_{i\in I}\kappa_{\Bdd(X)}(a_i)$.
For each $i\in I$, using that $\alpha_X\circ\kappa_{\Bdd(X)}=\gamma_X$, and using \autoref{pgr:alpha:bimod-BXYdd-full}, we get
\[
f \alpha_X(\kappa_{\Bdd(X)}(a_i))
= f \gamma_{X}(a_i)
= f\kappa_{\Bdd(X)}(a_i).
\]

Let $n\colon \Bdd(X)\times(X\tensProj Y^*)\to (X\tensProj Y^*)$ be the bilinear map implementing the left action of $\Bdd(X)$ on $X\tensProj Y^*$.
Identifying the dual of $X\tensProj Y^*$ with $\Bdd(X,Y^{**})$, we have $fa^{**}=\rtArensII{n}(f,a^{**})$;
see \autoref{dfn:actionsDualMod}.
A similar argument as in \autoref{rmk:alpha:BAlgBXXdd} shows that the map $\Bdd(X,X^{**})\to\Bdd(X,Y^{**})$ given by $h\mapsto fh$ is \weakStar{} continuous.
Using this and using that $\alpha_X$ is \weakStar{} continuous at the last step, and using at the first step that $\rtArensII{n}$ is \weakStar{} continuous in 
the second variable, we get
\[
f\alpha_X(a^{**})
= \wkStar\lim_{i\in I} f \alpha_X(\kappa_{\Bdd(X)}(a_i))
= \wkStar\lim_{i\in I} f\kappa_{\Bdd(X)}(a_i)
= fa^{**}.
\]

Write $\gamma$ for $\gamma_{X,Y}$,
Choose a net $(b_j)_{j\in J}$ in $\Bdd(Y)$ with $b^{**}=\wkStar\lim_{j\in J}\kappa_{\Bdd(Y)}(b_j)$.
Given $j\in J$, we deduce as above that
\[
\alpha_Y(\kappa_{\Bdd(Y)}(b_j)) \gamma(g)
= \gamma_{Y}(b_j)\gamma(g)
= \kappa_{\Bdd(Y)}(b_j)\gamma(g).
\]

Let $m\colon (X\tensProj Y^*)\times\Bdd(Y)\to X\tensProj Y^*$ be the bilinear map implementing the right action of $\Bdd(Y)$ on $X\tensProj Y^*$.
Then $b^{**}\gamma(g)=\ltArensII{m}(b^{**},\gamma(g))$.
As above, we get that the map $\Bdd(Y,Y^{**})\to\Bdd(X,Y^{**})$ given by $h\mapsto h\gamma(g)$ is \weakStar{} continuous.
Thus
\[
\alpha_Y(b^{**})\gamma(g)
= \wkStar\lim_{i\in I} \alpha_Y(\kappa_{\Bdd(Y)}(b))\gamma(g)
= \wkStar\lim_{j\in J} \kappa_{\Bdd(Y)}(b_j)\gamma(g)
= b^{**}\gamma_{X,Y}(g). \qedhere
\]
\end{proof}

\begin{lma}
\label{prp:alpha:alpha-bimod-full}
The map $\omega_{X,Y}\colon X\tensProj Y^* \to\Bdd(X,Y)^*$ from \autoref{pgr:alpha:omegaXY} is a $\Bdd(X)$-$\Bdd(Y)$-bimodule map.
Further, the map $\alpha_{X,Y}\colon\Bdd(X,Y)^{**}\to\Bdd(X,Y^{**})$ is a left $(\Bdd(Y)^{**},\ltArensProd)$-module map and also a right
$(\Bdd(X)^{**},\rtArensProd)$-module map.
\end{lma}
\begin{proof}
It is easy to see that $\omega_{X,Y}$ is a $\Bdd(X)$-$\Bdd(Y)$-bimodule map.
The second assertion follows from \autoref{prp:dualMod:extDualModFunctorial}, since $\alpha_{X,Y}=\omega_{X,Y}^*$.
\end{proof}

\begin{thm}
\label{prp:alpha:alpha-bimod-reduced}
Let $X$ and $Y$ be Banach spaces, let $a^{**}\in\Bdd(X)^{**}$, and let $f^{**}\in\Bdd(X,Y)^{**}$.
Then
\[
\alpha_{X,Y}(f^{**}\rtArensProd a^{**}) = \alpha_{X,Y}(f^{**})\alpha_X(a^{**}).
\]

If $Y$ is reflexive, then also $\alpha_{X,Y}(b^{**}\ltArensProd f^{**})=\alpha_Y(b^{**})\alpha_{X,Y}(f^{**})$ for $b^{**}\in\Bdd(Y)^{**}$.
\end{thm}
\begin{proof}
Use \autoref{prp:alpha:alpha-bimod-full} at the first step, and \autoref{prp:alpha:bimod-BXYdd-relation} at the second, to get
\[
\alpha_{X,Y}(f^{**}\rtArensProd a^{**}) = \alpha_{X,Y}(f^{**})a^{**} = \alpha_{X,Y}(f^{**})\alpha_X(a^{**}).
\]

Next, assume that $Y$ is reflexive, and let $b^{**}\in\Bdd(Y)^{**}$.
Since $Y$ is reflexive, the map $\gamma_{X,Y}$ is surjective and thus $\alpha_{X,Y}(f^{**})$ lies in its image.
Hence, \autoref{prp:alpha:bimod-BXYdd-relation} shows that $b^{**}\alpha_{X,Y}(f^{**}) = \alpha_Y(b^{**})\alpha_{X,Y}(f^{**})$.
Using this at the second step, and \autoref{prp:alpha:alpha-bimod-full} at the first step, we conclude that
\[
\alpha_{X,Y}(b^{**}\ltArensProd f^{**}) = b^{**}\alpha_{X,Y}(f^{**}) = \alpha_Y(b^{**})\alpha_{X,Y}(f^{**}).\qedhere
\]
\end{proof}


As an application, we prove a characterization of reflexivity in terms of the map $\alpha_X$. 
We thank Matthew Daws for providing us a proof of the converse implication.

\begin{prp}
\label{prp:alpha:alphaBXX-firstArensMultiplicative}
Let $X$ be a Banach space.
Then $X$ is reflexive if and only if $\alpha_X\colon(\Bdd(X)^{**},\ltArensProd)\to\Bdd(X,X^{**})$ is multiplicative.
\end{prp}
\begin{proof}
The `only if' implication follows from \autoref{prp:alpha:alpha-bimod-reduced}.
To prove the backward implication, assume that $X$ is not reflexive.
Applying \cite[Proposition~7]{Daw04ArensRegOpOnBSp}, choose bounded sequences $(x_n)_{n\in\NN}$ in $X$ and $(x^*_m)_{m\in\NN}$ in $X^*$ with
\[
\langle x^*_n, x_m \rangle
= \begin{cases}
1, &n\leq m \\
0, &n>m
\end{cases}.
\]
For $n,m\in\NN$, set $a_n = \theta_{x_0,x^*_n}$ and $b_m = \theta_{x_m,x^*_0}$.
Then $(a_n)_{n\in\NN}$ and $(b_m)_{m\in\NN}$ are bounded in $\Bdd(X)$.
Choose subnets $(a_{n_i})_{i\in I}$ and $(b_{m_j})_{j\in J}$ whose images in $\Bdd(X)^{**}$ converge \weakStar{}, and let $a^{**}$ and $b^{**}$ be their limits.
We have
\[
a^{**}\ltArensProd b^{**}
= \wkStar\lim_{i\in I}\kappa_{\Bdd(X)}(a_{n_i})b^{**}
= \wkStar\lim_{i\in I}\wkStar\lim_{j\in J} \kappa_{\Bdd(X)}(a_{n_i}b_{m_j}).
\]
Therefore
\begin{align*}
\langle \alpha_X(a^{**}\ltArensProd b^{**}),x_0\otimes x^*_0 \rangle
&= \langle a^{**}\ltArensProd b^{**}, \omega_X(x_0\otimes x^*_0) \rangle \\
&= \lim_{i\in I}\lim_{j\in J} \langle \kappa_{\Bdd(X)}(a_{n_i}b_{m_j}), \omega_X(x_0\otimes x^*_0) \rangle \\
&= \lim_{i\in I}\lim_{j\in J} \langle a_{n_i}b_{m_j}x_0, x^*_0 \rangle
= \lim_{i\in I}\lim_{j\in J} \langle x^*_{n_i}, x_{m_j} \rangle
= 1.
\end{align*}

Using that $\alpha_X\circ\kappa_{\Bdd(X)}=\gamma_X$ (\autoref{pgr:alpha:charAlpha}), and using that $\Bdd(X,X^{**})$ is a left dual Banach algebra and that right multiplication by an element in the image of $\gamma_X$ is \weakStar{} continuous (\autoref{rmk:alpha:BAlgBXXdd}), we deduce that
\[
\alpha_X(a^{**})\alpha_X(b^{**})
= \wkStar\lim_{j\in J} \alpha_X(a^{**})\gamma_X(b_{m_j})
= \wkStar\lim_{j\in J} \wkStar\lim_{i\in I} \gamma_X(a_{n_i}b_{m_j}).
\]
Hence,
\begin{align*}
\langle \alpha_X(a^{**})\alpha_X(b^{**}),x_0\otimes x^*_0 \rangle
&= \lim_{j\in J}\lim_{i\in I} \langle \gamma_X(a_{n_i}b_{m_j}), x_0\otimes x^*_0 \rangle \\
&= \lim_{j\in J}\lim_{i\in I} \langle a_{n_i}b_{m_j}x_0, x^*_0 \rangle
= \lim_{j\in J}\lim_{i\in I} \langle x^*_{n_i}, x_{m_j} \rangle
= 0,
\end{align*}
which shows that $\alpha_X(a^{**}\ltArensProd b^{**})\neq\alpha_X(a^{**})\alpha_X(b^{**})$, as desired.
\end{proof}

By \autoref{prp:alpha:alphaBXX-firstArensMultiplicative}, when $X$ is reflexive then $\alpha_X\colon \Bdd(X)^{**}\to\Bdd(X)$
is multiplicative for both Arens products on $\Bdd(X)^{**}$.
In \autoref{prp:refl:charXRefl_BX}, we prove a converse to this statement:
If there exists a map $r\colon\Bdd(X)^{**}\to\Bdd(X)$ that is multiplicative for either Arens product and satisfies $r\circ\kappa_{\Bdd(X)}=\id_{\Bdd(X)}$, then $X$ is reflexive.

\section{Complementation of \texorpdfstring{$\Bdd(X,Y)$}{L(X,Y))} in its bidual}
\label{sec:compl}

The main result of this section is \autoref{prp:compl:complBidual}, where we show that there is a natural one-to-one correspondence between projections from $Y^{**}$ onto $Y$ and projections from $\Bdd(X,Y)^{**}$ onto $\Bdd(X,Y)$ that are right $\Approx(X)$-module maps.
It follows that $Y$ is $\lambda$-complemented in its bidual\footnote{That is, there is a projection $\pi\colon Y^{**}\to Y$ with $\|\pi\|\leq\lambda$.} if and only if $\Bdd(X,Y)$ is $\lambda$-complemented in its bidual (as a Banach space or, equivalently, as a right $\Bdd(X)$-module);
see \autoref{prp:compl:charCompl}.

Throughout this section, $X$ and $Y$ denote Banach spaces.

\begin{dfn}
\label{pgr:compl:q_r_from_pi}
Let $\pi\colon Y^{**}\to Y$ be a bounded linear map.
We define the map $q_\pi\colon\Bdd(X,Y^{**})\to\Bdd(X,Y)$ by $q_\pi(f) = \pi\circ f$, for $f\in\Bdd(X,Y^{**})$, and we define $r_\pi\colon\Bdd(X,Y)^{**}\to\Bdd(X,Y)$ by $r_\pi = q_\pi\circ \alpha_{X,Y}$.
\end{dfn}

\begin{lma}
\label{prp:compl:q_r_from_pi}
Let $\pi\colon Y^{**}\to Y$ be a projection.
Then:
\begin{enumerate}
\item
$q_\pi$ is a right $\Bdd(X)$-module projection with $\|q_\pi\|\leq\|\pi\|$.
\item
$r_\pi$ is a right $\Bdd(X)$-module projection with $\|r_\pi\|\leq\|\pi\|$.
\end{enumerate}
\end{lma}
\begin{proof}
(1) is straightforward to check.
Let us verify~(2).
Using \autoref{pgr:alpha:charAlpha} at the second step, and using~(1) at the third step, we get
\[
r_\pi\circ\kappa_{\Bdd(X)}
= q_\pi\circ\alpha_{X,Y}\circ\kappa_{\Bdd(X)}
= q_\pi\circ\gamma_{X,Y}
= \id_{\Bdd(X)}.
\]
The estimate for $\|r_\pi\|$ follows easily using that $\alpha_{X,Y}$ is contractive.
Lastly, the map $r_\pi$ is a right $\Bdd(X)$-module map since both $\alpha_{X,Y}$ and $q_\pi$ are right $\Bdd(X)$-module maps, by part~(1) and \autoref{prp:alpha:alpha-bimod-full}.
\end{proof}

Recall that $f^\tr$ denotes the transpose of a map $f$.

\begin{dfn}
\label{pgr:compl:pi_from_r}
Let $r\colon\Bdd(X,Y)^{**}\to\Bdd(X,Y)$ be a bounded linear map, let $x\in X$ and let $x^*\in X^*$.
Let $\Theta_{x^*}\colon Y\to\Bdd(X,Y)$ and $\ev_x\colon\Bdd(X,Y)\to Y$ be as in \autoref{pgr:alpha:theta}.
We define the map $\pi_{r,x,x^*}\colon Y^{**}\to Y$ by $\pi_{r,x,x^*}=\ev_x\circ r\circ\Theta_{x^*}^{\tr\tr}$.
\end{dfn}

\begin{lma}
\label{prp:compl:pi_from_r}
Let $r\colon\Bdd(X,Y)^{**}\to\Bdd(X,Y)$ be a projection.
Let $x\in X$ and $x^*\in X^*$ with $\|x\|=\|x^*\|=\langle x,x^*\rangle=1$.
Then:
\begin{enumerate}
\item
We have $\pi_{r,x,x^*}\circ\kappa_Y=\id_Y$ and $\|\pi_{r,x,x^*}\|\leq\|r\|$.
\item
Let $r_{x,x^*}\colon\Bdd(X,Y)^{**}\to\Bdd(X,Y)$ be the map associated to $\pi_{r,x,x^*}$ as in \autoref{pgr:compl:q_r_from_pi}.
Then $r_{x,x^*}(f^{**})(z) = r( f^{**}\theta_{z,x^*})(x)$ for every $f^{**}\in\Bdd(X,Y)^{**}$ and $z\in X$.
In particular, $r=r_{x,x^*}$ if $r$ is a right $\Approx(X)$-module map.
\item
If $r$ is a right $\Approx(X)$-module map, then $\pi_{r,x,x^*}$ is independent of the choice of $x$ and $x^*$.
\end{enumerate}
\end{lma}
\begin{proof}
(1).
We have $\ev_x\circ \Theta_{x^*}=\id_Y$, since $\langle x,x^*\rangle=1$.
Using this at the last step, using that $\Theta_{x^*}^{\tr\tr}\circ\kappa_Y=\kappa_{\Bdd(X,Y)}\circ\Theta_{x^*}$ at the second step, and using that $r\circ\kappa_{\Bdd(X,Y)}=\id_{\Bdd(X,Y)}$ at the third step, we get
\[
\pi_{r,x,x^*}\circ\kappa_Y
= \ev_x\circ r\circ\Theta_{x^*}^{\tr\tr}\circ\kappa_Y
= \ev_x\circ r\circ\kappa_{\Bdd(X,Y)}\circ\Theta_{x^*}
= \ev_x\circ\Theta_{x^*}
= \id_Y.
\]
The estimate for the norm of $\pi_{x,x^*}$ follows easily using $\|\ev_x\|=\|\Theta_{x^*}\|=1$.

(2).
For every $f\in\Bdd(X,Y)$ and $z\in X$, we have
\[
(\Theta_{x^*}\circ\ev_z)(f)
= \Theta_{x^*}(f(z))
= \theta_{f(z),x^*}
= f\circ\theta_{z,x^*}.
\]
It follows that $\left(\Theta_{x^*}^{\tr\tr}\circ\ev_z^{\tr\tr}\right)(f^{**})=f^{**}\theta_{z,x^*}$ for every $f^{**}\in\Bdd(X,Y)^{**}$ and $z\in X$.
Using this at the last step, and using \autoref{pgr:alpha:charAlpha} at the second step, we get
\begin{align*}
r_{x,x^*}(f^{**})(z)
&= \pi_{r,x,x^*}(\alpha_{X,Y}(f^{**})(z))
= (\ev_x\circ r\circ\Theta_{x^*}^{\tr\tr})(\ev_z^{\tr\tr}(f^{**})) \\
&= r(\Theta_{x^*}^{\tr\tr}\circ\ev_z^{\tr\tr}(f^{**}))(x)
= r\left( f^{**}\theta_{z,x^*} \right)(x).
\end{align*}

Assume now that $r$ is a right $\Approx(X)$-module map.
Then
\[
r_{x,x^*}(f^{**})(z)
= r\left( f^{**}\theta_{z,x^*} \right)(x)
= r(f^{**})(\theta_{z,x^*}(x))
= r(f^{**})(z),
\]
for every $f^{**}\in\Bdd(X,Y)^{**}$ and $z\in X$.
We conclude that $r=r_{x,x^*}$, as desired.

(3).
Let $z\in X$ and $z^*\in X^*$ satisfy $\langle z,z^*\rangle=1$.
We have
\[
\Theta_{z^*}(y)\theta_{z,x^*}
= \theta_{y,z^*}\theta_{z,x^*}
= \theta_{y,x^*}
= \Theta_{x^*}(y),
\]
for all $y\in Y$, and hence $\Theta_{z^*}^{\tr\tr}(y^{**})\theta_{z,x^*}=\Theta_{x^*}^{\tr\tr}(y^{**})$ for all $y^{**}\in Y^{**}$.
Using this at the second step, using that $r$ is a right $\Approx(X)$-module map at the third step, and using that $\ev_x(a\theta_{z,x^*})=\ev_{z}(a)$ for all $a\in\Bdd(X)$ at the fourth step, we get
\begin{align*}
\pi_{r,x,x^*}(y^{**})
&= \ev_x \big( r(\Theta_{x^*}^{\tr\tr}(y^{**}) \big)
= \ev_x \big( r(\Theta_{z^*}^{\tr\tr}(y^{**})\theta_{z,x^*}) \big) \\
&= \ev_x \big( r(\Theta_{z^*}^{\tr\tr}(y^{**}))\theta_{x',\eta} \big) 
= \ev_{z}\big( r(\Theta_{z^*}^{\tr\tr}(y^{**})) \big)
= \pi_{r,z,z^*}(y^{**}),
\end{align*}
for every $y^{**}\in Y^{**}$.
It follows that $\pi_{r, x,x^*}=\pi_{r, z,z^*}$, as desired.
\end{proof}

\begin{dfn}
\label{dfn:compl:pi_from_r}
Assume that $X\neq\{0\}$.
Let $r\colon\Bdd(X,Y)^{**}\to\Bdd(X,Y)$ be a right $\Approx(X)$-module projection.
In view of part~(3) of \autoref{prp:compl:pi_from_r}, we define $\pi_r=\ev_x\circ r\circ\Theta_{x^*}^{\tr\tr}$, for any choice of $x\in X$ and $x^*\in X^*$ satisfying $\|x\|=\|x^*\|=\langle x,x^*\rangle=1$.
\end{dfn}

\begin{lma}
\label{prp:compl:alpha_of_omega}
Let $x^*\in X^*$ and $y^{**}\in Y^{**}$.
Then $\alpha_{X,Y} \left( \Theta_{x^*}^{\tr\tr}(y^{**}) \right)
=\theta_{y^{**},x^*}$.
\end{lma}
\begin{proof}
For $z\in X$ and $y\in Y$, we have
\[
(\ev_z\circ\Theta_{x^*})(y)
= \theta_{y,x^*}(z)
= \langle x^*,z\rangle y.
\]
It follows that $(\ev_z^{\tr\tr}\circ\Theta_{x^*}^{\tr\tr})(y^{**})=\langle x^*,z\rangle y^{**}$.
Using this at the second step, and using \autoref{pgr:alpha:charAlpha} at the first step, we conclude that
\[
\alpha_{X,Y}\left( \Theta_{x^*}^{\tr\tr}(y^{**}) \right)(z)
= \ev_z^{\tr\tr}\left( \Theta_{x^*}^{\tr\tr}(y^{**}) \right)
= \langle x^*,z\rangle y^{**}
= \theta_{y^{**},x^*}(z).\qedhere
\]
\end{proof}

The following theorem is the main result of this section and the foundation of most results in Sections~\ref{sec:predualBXY} and~\ref{sec:refl}.
Indeed, it is the essential new ingredient to prove the most interesting applications of this paper: Corollaries~\ref{prp:refl:charXRefl_BX} and~\ref{prp:refl:unique_BX}.

\begin{thm}
\label{prp:compl:complBidual}
Let $X$ and $Y$ be Banach spaces with $X\neq\{0\}$.
Assigning to $\pi\colon Y^{**}\to Y$ the maps $q_\pi$ and $r_\pi$ from \autoref{pgr:compl:q_r_from_pi} implements natural one-to-one correspondences between the following classes:
\begin{enumerate}
\item[(a)]
Projections $\pi\colon Y^{**}\to Y$.
\item[(b)]
Right $\Approx(X)$-module projections $q\colon\Bdd(X,Y^{**})\to\Bdd(X,Y)$.
\item[(c)]
Right $\Approx(X)$-module projections $r\colon\Bdd(X,Y)^{**}\to\Bdd(X,Y)$.
\end{enumerate}
Given a map $q$ as in~(b), the corresponding map in~(c) is $q\circ\alpha_{X,Y}$.
Given a map $r$ as in~(c), the corresponding map in~(a) is $\pi_r$ from \autoref{dfn:compl:pi_from_r}.
Moreover:
\begin{enumerate}
\item
For any map $\pi$ as in~(a), we have $\|\pi\| = \|q_\pi\| = \|r_\pi\|$.
\item
Every map $q$ as in~(b), and every map $r$ as in~(c), is automatically a right $\Bdd(X)$-module map.
\item
The kernel of $\pi$ is \weakStar{} closed if and only if the kernel of $q_\pi$ is \weakStar{} closed, if and only if the kernel of $r_\pi$ is \weakStar{} closed.
\end{enumerate}
\end{thm}
\begin{proof}
Throughout the proof, we write $\alpha$ for $\alpha_{X,Y}$.
We let $\mathfrak{A}$, $\mathfrak{B}$ and $\mathfrak{C}$ denote the sets of maps as in~(a), (b), and (c), respectively.
Given $\pi\in \mathfrak{A}$, it follows from \autoref{prp:compl:q_r_from_pi} that $q_\pi\in \mathfrak{B}$ and $r_\pi\in \mathfrak{C}$.
Given $q\in \mathfrak{B}$, we have $q\circ\alpha\in\mathfrak{C}$ because $\alpha$ is a right $\Approx(X)$-module map satisfying $\alpha\circ\kappa_{\Bdd(X)}=\gamma_{X,Y}$;
see Lemmas~\ref{pgr:alpha:charAlpha} and~\ref{prp:alpha:alpha-bimod-full}.

Given $r\in \mathfrak{C}$, we have $r=r_{\pi_r}$ by \autoref{prp:compl:pi_from_r}(2).
It remains to show that $\pi=\pi_{q_\pi\circ\alpha}$ and $q=q_{\pi_{q\circ\alpha}}$ for all $\pi\in \mathfrak{A}$ and $q\in \mathfrak{B}$.

Fix $x\in X$ and $x^*\in X^*$ satisfying $\|x\|=\|x^*\|=\langle x,x^*\rangle=1$.
Let $\pi\in \mathfrak{A}$ and $y^{**}\in Y^{**}$.
Using \autoref{prp:compl:alpha_of_omega} at the third step, we obtain
\begin{align*}
\pi_{q_\pi\circ\alpha}(y^{**})
&= (\ev_x\circ q_\pi\circ\alpha\circ\Theta_{x^*}^{\tr\tr})(y^{**})
= \pi\left( \alpha(\Theta_{x^*}^{\tr\tr}(y^{**})(x) \right)\\
&= \pi(\theta_{y^{**},x^*}(x)) 
= \pi(y^{**}),
\end{align*}
so $\pi=\pi_{q_\pi\circ\alpha}$.
Next, let $q\in \mathfrak{B}$, let $f\in\Bdd(X,Y^{**})$, and let $z\in X$.
Using that $q$ is a right $\Approx(X)$-module map at the sixth step, and using \autoref{prp:compl:alpha_of_omega} at the fourth step, we obtain
\begin{align*}
q_{\pi_{q\circ\alpha}}(f)(z)
&= \pi_{q\circ\alpha}(f(z))
= \left( \ev_x\circ q\circ\alpha\circ\Theta_{x^*}^{\tr\tr} \right)(f(z))
= q(\alpha(\Theta_{x^*}^{\tr\tr}(f(z))))(x)\\
&= q(\theta_{f(z),x^*})(x)
= q(f\theta_{z,x^*})(x)
= q(f)\theta_{z,x^*}(x)
= q(f)(z).
\end{align*}

(1).
Given $\pi\in \mathfrak{A}$, we have $\|q_\pi\|\leq\|\pi\|$ by \autoref{prp:compl:q_r_from_pi};
we have $\|r_\pi\|=\|q_\pi\circ\alpha\|\leq\|q_\pi\|$ since $\alpha$ is contractive;
and we have $\|\pi\|\leq\|r_\pi\|$ by \autoref{prp:compl:pi_from_r}.
It follows that $\pi$, $q_\pi$ and $r_\pi$ have the same norms.

(2).
This follows directly from \autoref{prp:compl:q_r_from_pi}.

(3).
Let $\pi\in \mathfrak{A}$ and assume that $\kernel(\pi)$ is \weakStar{} closed in $Y^{**}$.
Given $f\in\Bdd(X,Y^{**})$, we have $q_\pi(f)=0$ if and only if $f(x)\in\kernel(\pi)$ for every $x\in X$.
To show that $\kernel(q_\pi)$ is \weakStar{} closed in $\Bdd(X,Y^{**})$, let $(f_j)_{j\in J}$ be a net in $\kernel(q_\pi)$ that converges \weakStar{} to $f\in\Bdd(X,Y^{**})$.
Given $x\in X$, the net $(f_j(x))_{j\in J}$ converges \weakStar{} to $f(x)$ in $Y^{**}$.
By assumption, $f_j(x)\in\kernel(\pi)$ for each $j$.
Since $\kernel(\pi)$ is \weakStar{} closed, it follows that $f(x)\in\kernel(\pi)$.
Hence $f\in\kernel(q_\pi)$.

Next, if $\kernel(q_\pi)$ is \weakStar{} closed in $\Bdd(X,Y^{**})$, then $\kernel(r_\pi)$ is \weakStar{} closed in $\Bdd(X,Y)^{**}$ since $r_\pi=q_\pi\circ\alpha$ and $\alpha$ is \weakStar{} continuous.

Lastly, assume that $\kernel(r_\pi)$ is \weakStar{} closed in $\Bdd(X,Y)^{**}$.
Fix $x\in X$ and $x^*\in X^*$ satisfying $\|x\|=\|x^*\|=\langle x,x^*\rangle=1$.
Given $f^{**}\in\Bdd(X,Y)^{**}$, it is straightforward to check that $r_\pi(f^{**})(x)=0$ if and only if $r(f^{**}\theta_{x,x^*})=0$.
To show that $\kernel(\pi)$ is \weakStar{} closed in $Y^{**}$, let $(y_i^{**})_{i\in I}$ be a net in $\kernel(\pi)$ that converges \weakStar{} to $y^{**}\in Y^{**}$.
For each $i$, we have
$0 = \pi(y_i^{**}) = r_\pi(\Theta_{x^*}^{\tr\tr}(y_i^{**}))(x)$,
and therefore $r_\pi(\Theta_{x^*}^{\tr\tr}(y_i^{**})\theta_{x,x^*})=0$.
Using that $\Theta_{x^*}^{\tr\tr}(y_i^{**})\theta_{x,x^*}$ converges \weakStar{} to $\Theta_{x^*}^{\tr\tr}(y^{**})\theta_{x,x^*}$, and using that $\kernel(r_\pi)$ is \weakStar{} closed, we deduce that $r_\pi(\Theta_{x^*}^{\tr\tr}(y^{**})\theta_{x,x^*})=0$.
Hence $\pi(y^{**})=0$, as desired.
\end{proof}

\begin{cor}
\label{prp:compl:charCompl}
Let $X$ and $Y$ be Banach spaces with $X\neq\{0\}$, and let $\lambda\in\RR$ with $\lambda\geq 1$.
Then the following are equivalent:
\begin{enumerate}
\item
$Y$ is $\lambda$-complemented in its bidual as a Banach space.
\item
$\Bdd(X,Y)$ is $\lambda$-complemented in its bidual as a Banach space.
\item
$\Bdd(X,Y)$ is $\lambda$-complemented in its bidual as a right $\Bdd(X)$-module.
\item
$\Bdd(X,Y)$ is $\lambda$-complemented in $\Bdd(X,Y^{**})$ as a right $\Bdd(X)$-module (or, equivalently, as a right $\Approx(X)$-module).
\end{enumerate}
\end{cor}
\begin{proof}
The equivalence of~(1), (3) and~(4) follows from \autoref{prp:compl:complBidual}.
That~(3) implies~(2) is obvious.
Finally, that~(2) implies~(1) follows from \autoref{prp:compl:pi_from_r}.
\end{proof}

\section{Preduals of \texorpdfstring{$\Bdd(X,Y)$}{L(X,Y)}}
\label{sec:predualBXY}

The main result of this section, \autoref{prp:predualBXY:correspondencePreduals}, asserts that there is a natural
one-to-one correspondence between (isometric) preduals of $Y$ and (isometric) preduals of $\Bdd(X,Y)$ making it a right dual $\Bdd(X)$-module.
In proving this result, the map $\alpha_{X,Y}$ from \autoref{dfn:alpha:alphaXY} will be crucial.

Throughout this section, $X$ and $Y$ denote Banach spaces.

\begin{dfn}
\label{dfn:predualBXY:predualBXY}
Let $F\subseteq Y^*$ be a closed subspace with inclusion map $\iota_F\colon F\to Y^*$.
We set
\[
\beta_{X,F}:= \omega_{X,Y}\circ(\id_X\tensProj\iota_F)\colon X\tensProj F\to\Bdd(X,Y)^*,
\]
and we define $X\tensPredual{X}{Y}F$ as the image of $\beta_{X,F}$.
\end{dfn}

\begin{lma}
\label{prp:predualBXY:predualBXYIso}
Let $F\subseteq Y^*$ be a predual.
Then $\beta_{X,F}\colon X\tensProj F\to X\tensPredual{X}{Y}F$ is a contractive isomorphism satisfying $\|\beta_{X,F}^{-1}\|\leq \|\delta_F^{-1}\|$.
In particular, $X\tensPredual{X}{Y}F$ is a closed subspace of $\Bdd(X,Y)^*$.
\end{lma}
\begin{proof}
Clearly $\beta_{X,F}$ is a contractive, linear map.
Let us show that $\beta_{X,F}$ is bounded below by $\|\delta_F^{-1}\|^{-1}$.
Let $t=\sum_{k=1}^n x_k\otimes y_k^* \in X\widehat{\otimes} F$ be a sum of simple tensors, and let $\varepsilon>0$.
Recall that $\|t\|_\pi$ denotes the projective tensor norm of $t$.
Choose $t^*\in(X\tensProj F)^*$ with $\|t^*\|\leq 1$ and $|\langle t^*,t\rangle|\geq\|t\|_\pi-\varepsilon$.
Using the natural isometric isomorphism $(X\tensProj F)^*\cong\Bdd(X,F^*)$ from \autoref{pgr:prelim:tensProj}, the functional $t^*$ corresponds to a contractive operator $g\in\Bdd(X,F^*)$ satisfying
\[
|\langle g,t\rangle|
= \left| \sum_{k=1}^n \langle g(x_k), y_k^* \rangle \right|
\geq\|t\|_\pi-\varepsilon.
\]
As in \autoref{dfn:predualBSp:concretePredualBSp}, set $\delta_F=\iota_F^\tr\circ\kappa_Y\colon Y\to F^*$.
Set $h:=\delta_F^{-1}\circ g\in\Bdd(X,Y)$.
Given $y^*\in F\subseteq Y^*$ and $y^{**}\in F^*\subseteq Y^{**}$, it is easy to check that $\langle y^{**},y^*\rangle=\langle\delta_F^{-1}(y^{**}),\iota_F(y^*)\rangle$.
Using this at the third step, we obtain
\begin{align*}
|\langle h,\beta_{X,F}(t)\rangle|
&= \left| \sum_{k=1}^n \left\langle h(x_k), \iota_F(y_k^*) \right\rangle \right|
= \left| \sum_{k=1}^n \left\langle \delta_F^{-1}(g(x_k)), \iota_F(y_k^*) \right\rangle \right| \\
&= \left| \sum_{k=1}^n \left\langle g(x_k), y_k^* \right\rangle \right|
= |\langle g,t\rangle|.
\end{align*}

Since $\|h\|\leq\|\delta_F^{-1}\|$, it follows that
\[
\|\beta_{X,F}(t)\|
\geq \|h\|^{-1} |\langle h,\beta_{X,F}(t)\rangle|
= \|h\|^{-1} |\langle g,t\rangle|
\geq \|\delta_F^{-1}\|^{-1} (\|t\|_\pi-\varepsilon).
\]
We conclude that $\|\beta_{X,F}(t)\| \geq  \|\delta_F^{-1}\|^{-1} \|t\|_\pi$, as desired.
\end{proof}

\begin{lma}
\label{prp:predualBXY:sliceInduced}
Let $x^*\in X^*$ with $\|x^*\|\leq 1$, let $F\subseteq Y^*$ be a predual, and let $R_{x^*}\colon X\tensProj F\to F$ be the map from \autoref{pgr:alpha:sliceMaps}.
Then $\left\| R_{x^*}\circ\beta_{X,F}^{-1} \right\| \leq 1$.
\end{lma}
\begin{proof}
Let $t=\sum_{k=1}^n x_k\otimes y_k^*\in X\widehat{\otimes} F$ be a sum of simple tensors.
Then $R_{x^*}(t)=\sum_{k=1}^n \langle x^*,x_k\rangle y_k^*$, by definition.
The result follows from the following computation, where we use the inclusion $F\subseteq Y^*$ at the first step, that $\|\theta_{y,x^*}\|=\|y\|\|x^*\|$ for every $y\in Y$ at the fourth step,
and the definition of the norm on $X\tensPredual{X}{Y}F$ at the last step:
\begin{align*}
\|R_{x^*}(t)\|
&= \sup \big\{ |\langle R_{x^*}(t), y \rangle| : y\in Y, \|y\|\leq 1 \big\} \\
&= \sup \left\{ \left|\sum_{k=1}^n \langle x^*,x_k\rangle \langle y_k^* , y \rangle \right| : y\in Y, \|y\|\leq 1 \right\} \\
&= \sup \left\{ \left|\sum_{k=1}^n \langle\theta_{y,x^*}(x_k), y_k^* \rangle \right| : y\in Y, \|y\|\leq 1 \right\} \\
&\leq \sup \left\{ \left|\sum_{k=1}^n\langle f(x_k), y_k^* \rangle \right| : f\in\Bdd(X,Y), \|f\|\leq 1 \right\}
= \|t\|_{X\tensPredual{X}{Y}F}.\qedhere
\end{align*}
\end{proof}

Given a linear map $\delta$, recall that $\delta_*$ denotes postcomposition with $\delta$.

\begin{thm}
\label{prp:predualBXY:predualBXY}
Let $X$ and $Y$ be Banach spaces, and let $F\subseteq Y^*$ be a predual.
Then $X\tensPredual{X}{Y}F$ is a predual of $\Bdd(X,Y)$ satisfying
$\|\delta_F^{-1}\|=\|\delta_{ X\tensPredual{X}{Y}F}^{-1}\|$.
Moreover, after identifying $\Bdd(X,F^*)$ with $(X\tensProj F)^*$, we have
\begin{align}
\label{prp:predualBXY:predualBXY:eqThm}
\delta_{X\tensPredual{X}{Y}F}
= (\beta_{X,F}^\tr)^{-1}\circ(\delta_F)_*, \andSep
\delta_{X\tensPredual{X}{Y}F}^{-1}
= (\delta_F^{-1})_* \circ \beta_{X,F}^\tr.
\end{align}

Finally, $(\Bdd(X,Y),X\tensPredual{X}{Y}F)$ is a right dual $\Bdd(X)$-module.
\end{thm}
\begin{proof}
Let $\iota\colon X\tensPredual{X}{Y}F \to \Bdd(X,Y)^*$ denote the inclusion map, and set $\delta_{X\tensPredual{X}{Y}F} := \iota^\tr\circ\kappa_{\Bdd(X,Y)}$.
We need to show that $\delta_{X\tensPredual{X}{Y}F}$ is an isomorphism.
Consider the following diagram, which is easily checked to commute:
\[
\xymatrix@R-10pt{
(X\tensProj F)^*
& (X\tensPredual{X}{Y}F)^* \ar[l]_-{\beta_{X,F}^\tr}
& {\Bdd(X,Y)^{**}} \ar@{->>}[l]_-{\iota^\tr}  \\
\Bdd(X,F^*) \ar@{}[u]|-{\vcong}
& & \Bdd(X,Y) \ar[ll]^-{(\delta_F)_*} \ar[u]_-{\kappa_{\Bdd(X,Y)}} \ar[ul]^-{\delta_{X\tensPredual{X}{Y}F}}
.}
\]

It follows from \autoref{prp:predualBXY:predualBXYIso} that $\beta_{X,F}^\tr$ is a contractive isomorphism.
Since $(\delta_F)_*$ is also an isomorphism, we conclude that $\delta_{X\tensPredual{X}{Y}F}$ is an isomorphism satisfying \eqref{prp:predualBXY:predualBXY:eqThm}.
We deduce that
\[
\left\| \delta_{X\tensPredual{X}{Y}F}^{-1} \right\|
= \left\| (\delta_F^{-1})_* \circ \beta_{X,F}^\tr \right\|
\leq \|(\delta_F^{-1})_*\|
= \|\delta_F^{-1}\|.
\]

To prove the inverse inequality, choose $x\in X$ and $x^*\in X^*$ with $\|x\|=\|x^*\|=\langle x^*,x\rangle=1$.
We claim that
\[
\delta_F^{-1}
= \ev_x\circ\delta_{X\tensPredual{X}{Y}F}^{-1}\circ(\beta_{X,F}^\tr)^{-1}\circ\Theta_{x^*}.
\]
Indeed, given $y^{**}\in F^*\subseteq Y^{**}$, we use \eqref{prp:predualBXY:predualBXY:eqThm} at the first step to get
\begin{align*}
\left( \ev_x\circ\delta_{X\tensPredual{X}{Y}F}^{-1}\circ(\beta_{X,F}^\tr)^{-1}\circ\Theta_{x^*} \right)(y^{**})
&= \left( \ev_x\circ(\delta_F^{-1})_*\circ\Theta_{x^*} \right)(y^{**}) \\
&= \left( \delta_F^{-1}\circ\theta_{y^{**},x^*} \right)(x)
= \delta_F^{-1}(y^{**}),
\end{align*}
proving the claim.
By \autoref{lma:evThetawkstercts}, we have $\Theta_{x^*}=R_{x^*}^\tr$.
Using this at the first step, and using \autoref{prp:predualBXY:sliceInduced} at the last step, we compute
\begin{align}
\left\| (\beta_{X,F}^\tr)^{-1}\circ\Theta_{x^*} \right\|
= \left\| (\beta_{X,F}^{-1})^\tr\circ R_{x^*}^\tr\right\|
= \left\| R_{x^*}\circ\beta_{X,F}^{-1}\right\|
\leq 1.
\end{align}
Since $\|\ev_x\|=\|x\|=1$, we conclude that
\begin{align*}
\| \delta_F^{-1} \|
= \left\| \ev_x\circ\delta_{X\tensPredual{X}{Y}F}^{-1}\circ(\beta_{X,F}^\tr)^{-1}\circ\Theta_{x^*} \right\| 
\leq \left\| \delta_{X\tensPredual{X}{Y}F}^{-1} \right\|.
\end{align*}

Lastly, note that $\omega_{X,Y}\colon X\tensProj Y^*\to\Bdd(X,Y)^*$ is a $\Bdd(X)$-$\Bdd(Y)$-bimodule map by \autoref{prp:alpha:alpha-bimod-full},
so $X\tensPredual{X}{Y}F$ is a left $\Bdd(X)$-submodule of $\Bdd(X,Y)^*$.
Thus, it follows from \autoref{prp:predualMod:charDualMod} that $(\Bdd(X,Y),X\tensPredual{X}{Y}F)$ is a right dual $\Bdd(X)$-module.
\end{proof}

The following is our main result.
It asserts that the correspondences from \autoref{prp:predualBSP:summary} and \autoref{prp:compl:complBidual} induce natural bijections
between (isometric) preduals of $Y$ and (isometric) preduals of $\Bdd(X,Y)$ turning it into a right $\Bdd(X)$-module.
\begin{thm}
\label{prp:predualBXY:correspondencePreduals}
Let $X$ and $Y$ be Banach spaces with $X\neq\{0\}$.
Then there exist natural one-to-one correspondences between the following sets:
\begin{enumerate}
\item[(a)]
Concrete preduals of $Y$.
\item[(b)]
Concrete preduals of $\Bdd(X,Y)$ making it a right dual $\Approx(X)$-module.
\item[(c)]
Projections $\pi\colon Y^{**} \to Y$ with \weakStar{} closed kernel.
\item[(d)]
Right $\Approx(X)$-module projections $r\colon\Bdd(X,Y)^{**}\to\Bdd(X,Y)$ with \weakStar{} closed kernel.
\item[(e)]
Right $\Approx(X)$-module projections $q\colon\Bdd(X,Y^{**})\to\Bdd(X,Y)$ with \weakStar{} closed kernel.
\end{enumerate}

Given a predual $F\subseteq Y^*$,
the corresponding predual of $\Bdd(X,Y)$ is $X\tensPredual{X}{Y}F$;
the corresponding maps in~(c), (d) and~(e) are $\pi_F$, $r_{\pi_F}$ and $q_{\pi_F}$.
Moreover:
\begin{enumerate}
\item
For a predual $F\subseteq Y^*$, we have
\[
\|\delta_F^{-1}\|
= \|\delta_{X\tensPredual{X}{Y}F}^{-1}\|
= \| \pi_F \|
= \| r_{\pi_F} \|
= \| q_{\pi_F} \|.
\]
In particular, $F$ is isometric if and only if so is $X\tensPredual{X}{Y}F$, and if and only if any (and thus all) of 
$\pi_F$, $r_{\pi_F}$ or $q_{\pi_F}$ has norm one.
\item
Every predual as in~(b) automatically makes $\Bdd(X,Y)$ a right dual $\Bdd(X)$-module.
Every map $r$ as in~(d), and every map $q$ as in~(e), is automatically a right $\Bdd(X)$-module map.
\end{enumerate}
\end{thm}
\begin{proof}
The correspondence between~(a) and~(c) follows from \autoref{pgr:predualBSp:predualProj}.
The correspondence between~(c), (d) and~(e) follows from \autoref{prp:compl:complBidual}(3).
The correspondence between~(b), and~(d) follows from \autoref{prp:predualMod:charDualMod}.

Given a predual $F\subseteq Y^*$ as in~(a), it follows from \autoref{prp:predualBXY:predualBXY} that $X\tensPredual{X}{Y}F$ is a predual of $\Bdd(X,Y)^*$.
Let us show that $X\tensPredual{X}{Y}F$ is the `correct' predual as in~(b) corresponding to $F$.
Let $\iota_F\colon F\to Y^*$ denote the inclusion map, let $\id_X\tensProj\iota_F\colon X\tensProj F\to X\tensProj Y^*$ denote the induced map, and let $\iota\colon X\tensPredual{X}{Y}F\to\Bdd(X,Y)^*$ denote the inclusion map.
It is easy to check that
\[
\omega_{X,Y}\circ(\id_X\tensProj\iota_F) = \iota\circ\beta_{X,F}.\]
Since the transpose of $\id_X\tensProj\iota_F$ is $(\iota_F^\tr)_*\colon\Bdd(X,Y^{**})\to\Bdd(X,F^*)$, the following diagram commutes:
\[
\xymatrix@R-10pt{
\Bdd(X,Y^{**}) \ar[d]_-{(\iota_F^\tr)_*}
&& \Bdd(X,Y)^{**} \ar[ll]_-{\alpha_{X,Y}} \ar@{->>}[d]^-{\iota^\tr}
&& \Bdd(X,Y) \ar@{_{(}->}[ll]_-{\kappa_{\Bdd(X,Y)}} \ar[dll]^{\delta_{X\tensPredual{X}{Y}F}} \\
\Bdd(X,F^*)
&& (X\tensPredual{X}{Y}F)^* \ar[ll]^-{\beta_{X,F}^\tr}
.}
\]

We set $\pi_F=\delta_F^{-1}\circ\iota_F^\tr$, which is the map as in~(c) corresponding to $F$.
Then $(\pi_F)_*=q_{\pi_F}$.
The map as in~(d) corresponding to $\pi_F$ is $r_{\pi_F}=(\pi_F)_*\circ\alpha_{X,Y}$.
Considering the predual $X\tensPredual{X}{Y}F$ as in~(b), the corresponding map as in~(d) is $\pi_{X\tensPredual{X}{Y}F}=\delta_{X\tensPredual{X}{Y}F}^{-1} \circ\iota^\tr$.
Using \autoref{prp:predualBXY:predualBXY} at the second step, we deduce that
\begin{align*}
\pi_{X\tensPredual{X}{Y}F}
&= \delta_{X\tensPredual{X}{Y}F}^{-1} \circ\iota^\tr
= (\delta_F^{-1})_* \circ\beta_{X,F}^\tr \circ\iota^\tr
= (\delta_F^{-1})_* \circ(\iota_F^\tr)_* \circ\alpha_{X,Y} \\
&= (\delta_F^{-1}\circ\iota_F^\tr)_* \circ\alpha_{X,Y}
= (\pi_F)_* \circ\alpha_{X,Y}
= r_{\pi_F}.
\end{align*}

(1).
This follows directly from \autoref{prp:compl:complBidual}(1), \autoref{prp:predualBXY:predualBXY} and \autoref{prp:predualBSP:summary}.

(2).
It follows from \autoref{prp:compl:complBidual}(2) that the maps in~(d) and~(e) are automatically right $\Bdd(X)$-module maps.
Given a predual $G\subseteq\Bdd(X,Y)^*$ as in~(b), the associated map $\pi_G$ as in~(d)
is a right $\Approx(X)$-module map, and therefore automatically a right $\Bdd(X)$-module map.
It follows from \autoref{prp:predualMod:charDualMod} that the predual $G$ makes $\Bdd(X,Y)$ a right dual $\Bdd(X)$-module.
\end{proof}

For the case of strongly unique preduals (see \autoref{pgr:predualBSp:unique}), we immediately deduce the following result, which is new even in the case of isometric preduals.

\begin{cor}
\label{prp:predualBXY:uniquePreduals}
Let $Y$ be a Banach space.
Then the following are equivalent:
\begin{enumerate}
\item
$Y$ has a strongly unique (isometric) predual.
\item
For some (any) Banach space $X$ with $X\neq\{0\}$, the space $\Bdd(X,Y)$ has a strongly unique (isometric) predual making it a right dual $\Bdd(X)$-module.
\item
$\Bdd(Y)$ has a strongly unique (isometric) predual making it a right dual Banach algebra.
\end{enumerate}
\end{cor}

\section{Preduals of \texorpdfstring{$\Bdd(X,Y)$}{L(X,Y)} and \texorpdfstring{$\Bdd(Y)$}{L(Y)} for \texorpdfstring{$Y$}{Y} reflexive}
\label{sec:refl}

In this section, we first obtain characterizations of reflexivity of $Y$ in terms of properties of the $\Bdd(Y)$-$\Bdd(X)$-bimodule $\Bdd(X,Y)$;
see \autoref{prp:refl:charYRefl_BXY}.
Secondly, under the assumption that $Y$ be reflexive, we study uniqueness of preduals of $\Bdd(X,Y)$, and projections from $\Bdd(X,Y)^{**}$ onto $\Bdd(X,Y)$;
see \autoref{prp:refl:unique_BXY}.
Applied to the case $X=Y$, we extend results of Daws from \cite{Daw04PhD} and \cite{Daw07DualBAlgReprInj}.

\begin{lma}
\label{prp:refl:q_leftMod}
Let $X$ and $Y$ be Banach spaces with $X\neq\{0\}$, and assume that there exists a left $\Approx(Y)$-module projection $r\colon\Bdd(X,Y)^{**}\to\Bdd(X,Y)$.
Then $Y$ is reflexive.
\end{lma}
\begin{proof}
Choose $x\in X$ and $x^*\in X^*$ with $\langle x,x^*\rangle =1$.
Let $\Theta_{x^*}\colon Y\to\Bdd(X,Y)$ and $\ev_x\colon\Bdd(X,Y)\to Y$ be as in \autoref{pgr:alpha:theta}, and define $\pi\colon Y^{**}\to Y$ as $\pi:=\ev_x\circ r\circ\Theta_{x^*}^{\tr\tr}$. 
Then $\pi\circ\kappa_Y=\id_Y$ and thus, to check that $\kappa_Y$ is an isomorphism, it suffices to show that $\pi$ is injective.
Let $y^{**}\in Y^{**}$ be nonzero.
Choose $y^*\in Y^*$ with $\langle y^*, y^{**}\rangle=1$, and choose $y\in Y$ nonzero.
It is straightforward to verify that $\theta_{y,y^*}\Theta_{x^*}^{\tr\tr}(y^{**})=\kappa_{\Bdd(X,Y)}(\theta_{y,x^*})$.
Using this at the third step, and using at the second step that $r$ is a left $\Approx(Y)$-module map, we get
\begin{align*}
\theta_{y,y^*}[ \pi(y^{**}) ]
&= \theta_{y,y^*}[ r(\Theta_{x^*}^{\tr\tr}(y^{**}))(x) ] 
= r( \theta_{y,y^*} \Theta_{x^*}^{\tr\tr}(y^{**}) )(x) \\
&= r( \kappa_{\Bdd(X,Y)}(\theta_{y,x^*}) )(x) 
= \theta_{y,x^*}(x) = y \neq 0,
\end{align*}
which shows that $\pi(y^{**})\neq 0$, as desired.
\end{proof}

\begin{thm}
\label{prp:refl:charYRefl_BXY}
Let $X$ and $Y$ be Banach spaces with $X\neq\{0\}$.
Then the following are equivalent:
\begin{enumerate}
\item
$Y$ is reflexive.
\item
$\Bdd(X,Y)$ has a predual making it a dual $\Bdd(Y)$-$\Bdd(X)$-bimodule.
\item
$\Bdd(X,Y)$ has a predual making it a left dual $\Approx(Y)$-module.
\item
$\Bdd(X,Y)$ is complemented in its bidual as a $\Bdd(Y)$-$\Bdd(X)$-bimodule.
\item
$\Bdd(X,Y)$ is complemented in its bidual as a left $\Approx(Y)$-module.
\item
$\Bdd(X,Y)$ is complemented in $\Bdd(X,Y^{**})$ as a $\Bdd(Y)$-$\Bdd(X)$-bimodule.
\item
$\Bdd(X,Y)$ is complemented in $\Bdd(X,Y^{**})$ as a left $\Approx(Y)$-module.
\end{enumerate}

Moreover, in~(2) and~(3) we may equivalently assume that the predual is isometric, and in~(4)-(7) we may equivalently replace `complemented' by `$1$-complemented'.
\end{thm}
\begin{proof}
It is clear that~(1) implies~(6), that~(2) implies~(3), that~(4) implies~(5), and that~(6) implies~(7).
It follows from \autoref{prp:refl:q_leftMod} that~(5) implies~(1).

To show that ~(3) implies~(5), let $F\subseteq\Bdd(X,Y)^*$ be a predual making $\Bdd(X,Y)$ a left dual $\Approx(Y)$-module.
It follows from \autoref{prp:predualMod:charDualMod} that $\pi_F$ is a left $\Approx(Y)$-module map, as desired.
One shows that~(2) implies~(4) analogously.

To show that~(7) implies~(5), let $q\colon\Bdd(X,Y^{**})\to\Bdd(X,Y)$ be a left $\Approx(Y)$-module projection. 
Using that $\alpha_{X,Y}$ is a left $\Approx(Y)$-module map with $\alpha_{X,Y}\circ\kappa_{\Bdd(X,Y)}=\gamma_{X,Y}$ (see Lemmas~\ref{prp:alpha:alpha-bimod-full} and~\ref{pgr:alpha:charAlpha}), we get that $\Bdd(X,Y)$ is complemented in its bidual as a left $\Approx(Y)$-module via the map $q\circ\alpha_{X,Y}$.

It remains to show that~(1) implies~(2).
If $Y$ is reflexive,
then $(X\tensProj Y^*)^* \cong \Bdd(X,Y^{**}) \cong \Bdd(X,Y)$, and 
thus $X\tensProj Y^*$ is an isometric predual of $\Bdd(X,Y)$.
Then $X\tensProj Y^*$ makes $\Bdd(X,Y)$ a dual $\Bdd(Y)$-$\Bdd(X)$-bimodule by \autoref{prp:alpha:bimod-BXYd-wkStar}.
\end{proof}

As an application, we obtain several characterizations of reflexivity.
The equivalence between (1) and (2) has also been obtained by Daws in \cite[Proposition~4.2.1]{Daw04PhD}, while the remaining ones are new.
The statements in (4), (5) and (6) can be regarded as algebraic characterizations of reflexivity.

\begin{cor}
\label{prp:refl:charXRefl_BX}
Let $X$ be a Banach space.
Then the following are equivalent:
\begin{enumerate}
\item
$X$ is reflexive.
\item
$\Bdd(X)$ has a predual making it a dual Banach algebra.
\item
$\Bdd(X)$ has a predual making it a left dual Banach algebra.
\item
$\Bdd(X)$ is complemented in its bidual as a left $\Approx(X)$-module.
\item
There exists a projection $r\colon\Bdd(X)^{**}\to\Bdd(X)$ 
that is multiplicative for the left (equivalentely, for the right, or both) Arens product on $\Bdd(X)^{**}$.
\item
There exists a multiplicative projection $q\colon\Bdd(X,X^{**})\to\Bdd(X)$. 
\end{enumerate}
\end{cor}
\begin{proof}
The equivalence of~(1), (2), (3) and~(4) follows from \autoref{prp:refl:charYRefl_BXY} by considering $Y=X$.
It is clear that~(1) implies~(6), and that~(5) implies~(4).
By \autoref{prp:alpha:alpha-bimod-reduced}, (1) implies~(5).
That~(6) implies~(5) follows from \autoref{prp:alpha:alpha-bimod-reduced}.
\end{proof}

\begin{thm}
\label{prp:refl:unique_BXY}
Let $X$ and $Y$ be Banach spaces with $Y$ reflexive.
Then:
\begin{enumerate}
\item
$X\tensProj Y^*$ is the unique predual making $\Bdd(X,Y)$ a right dual $\Approx(X)$-module.
\item
$\alpha_{X,Y}\colon\Bdd(X,Y)^{**}\to\Bdd(X,Y)$ is the unique right $\Approx(X)$-module projection.
\end{enumerate}
\end{thm}
\begin{proof}
Since $Y$ is reflexive, it (obviously) has a strongly unique predual.
Hence, the first statement follows from the equivalence of (a) and (b) in \autoref{prp:predualBXY:correspondencePreduals}.
Similarly, there is a unique projection $Y^{**}\to Y$, whence the second statement follows from the equivalence of (a) and (c) in \autoref{prp:compl:complBidual}.
\end{proof}

\begin{cor}
\label{prp:refl:unique_BX}
Let $X$ be a reflexive Banach space.
Then:
\begin{enumerate}
\item
If $F\subseteq\Bdd(X)^*$ is a predual making $\Bdd(X)$ a right dual Banach algebra, then $F=X\tensProj X^*$.
In particular, $F$ is automatically an isometric predual making $\Bdd(X)$ a dual Banach algebra.
\item
If $r\colon\Bdd(X)^{**}\to\Bdd(X)$ is a right $\Approx(X)$-module projection, then $r=\alpha_X$.
In particular, $r$ is automatically a quotient map and multiplicative for both Arens products on $\Bdd(X)^{**}$.
\item
For every right dual Banach algebra $A$ and for every (not necessarily isometric) Banach algebra isomorphism $\varphi\colon\Bdd(X)\to A$, the maps $\varphi$ and $\varphi^{-1}$ are automatically \weakStar{} continuous.
\end{enumerate}
\end{cor}

Our \autoref{prp:refl:unique_BX} generalizes \cite[Theorem~4.4]{Daw07DualBAlgReprInj} in two ways.
First, we see that the assumption that $X$ have the approximation property is unnecessary.
Second, the predual is not only unique among the preduals making $\Bdd(X)$ a dual Banach algebra, but also among the preduals making it a right dual Banach algebra.


\begin{rmk}
\label{rmk:refl:GodSapUnique}
If $X$ and $Y$ are reflexive, then $X\tensProj Y^*$ is the strongly unique isometric predual of $\Bdd(X,Y)$ by \cite[Proposition~5.10]{GodSap88DualitySpOpsSmoothNorms}.
In particular, if $X$ is reflexive, then the Banach algebra $\Bdd(X)$ has a strongly unique isometric predual.
\end{rmk}

Thus, if $X$ is reflexive, then $X\tensProj X^*$ is both the strongly unique isometric predual of $\Bdd(X)$ and the strongly unique predual making $\Bdd(X)$ a right dual Banach algebra;
see \autoref{prp:refl:unique_BX}.
One might wonder if $X\tensProj X^*$  is even strongly unique.
The next example shows that this is not the case in general.

\begin{exa}
\label{exa:refl:notStrUnique}
We claim that $\Bdd(\ell_2)$ does not have a strongly unique predual (even though it has a strongle unique \emph{isometric} predual).

The action of $\ell_\infty$ on $\ell_2$ by pointwise multiplication defines an isometric map $\ell_\infty\to\Bdd(\ell_2)$ with \weakStar{} closed image.
Let $P\colon\Bdd(\ell_2)\to\ell_\infty$ be the projection onto the diagonal, 
and denote by $X$ the range of $(1-P)$, which equals the kernel of $P$ 
Since $P$ is \weakStar{}-continuous, it follows that $X$ is \weakStar{} closed in $\Bdd(\ell_2)$, and hence has a predual $F_X\subseteq X^*$.
It is well-known that $\ell_\infty$ does not have a strongly unique predual.
Indeed, on the one hand, $\ell_1$ is its canonical predual (which is actually its strongly unique isometric predual).
On the other hand, it was shown by Pelczynski that $\ell_\infty$ is isomorphic (but not isometrically) to $L_\infty([0,1])$.
The canonical predual of $L_\infty([0,1])$ is $L_1([0,1])$, so $\ell_\infty$ has a predual $Z$ that is isomorphic to $L_1([0,1])$, and thus not isomorphic to $\ell_1$; see page~245 of~\cite{Ban32TheoryLinOperations}.

We obtain two different preduals $\ell_1\oplus F_X$ and $Z\oplus F_X$ of $\ell_\infty\oplus X$.
Hence $\ell_\infty\oplus X$ does not have a strongly unique predual.
Using the isomorphism between $\Bdd(\ell_2)$ and $\ell_\infty\oplus X$, we conclude that $\Bdd(\ell_2)$ does not have a strongly unique predual.
\end{exa}

\appendix


\end{document}